\documentclass{article}
\usepackage{graphicx} 
\usepackage[margin=1in]{geometry}
\usepackage[T1]{fontenc}
\usepackage{times}
\usepackage{amsmath,amsfonts,amssymb,amsthm}
\usepackage{hyperref,cleveref,doi}
\hypersetup{colorlinks=true, breaklinks=true, urlcolor=black, linkcolor=black, citecolor=black}
\usepackage{color,enumerate,csquotes,enumitem}

\newtheorem{theorem}{Theorem}[section]
\newtheorem{lemma}[theorem]{Lemma}
\newtheorem{proposition}[theorem]{Proposition}
\newtheorem{corollary}[theorem]{Corollary}
\theoremstyle{remark}
\newtheorem{question}{Question}
\newtheorem{remark}{Remark}

\newlength{\bibitemsep}
\newlength{\bibparskip}
\let\oldthebibliography\thebibliography
\renewcommand\thebibliography[1]{%
  \oldthebibliography{#1}%
  \setlength{\parskip}{\bibitemsep}%
  \setlength{\itemsep}{\bibparskip}%
}

\def\N{\mathbb{N}}
\def\Z{\mathbb{Z}}
\def\R{\mathbb{R}}
\def\cH{\mathcal{H}}
\def\m{\!-\!}
\def\p{\!+\!}
\def\O{\mathcal{O}}
\newcommand{\cv}[1]{\underset{#1}{\longrightarrow}}

\newcommand{\Landau}[3]{\underset{#2}{#1}\left(#3\right)} 
\newcommand{\landauT}[3]{\widetilde{#1}_{#2}\left(#3\right)} 
\newcommand{\abs}[1]{\left|#1\right|}
\newcommand{\card}[1]{\left|\left\{#1\right\}\right|}
\def\E{\mathbb{E}}
\newcommand{\expec}[1]{\mathbb{E}\left[#1\right]}
\newcommand{\expecond}[2]{\mathbb{E}\left[\left. #1 \,\right|\, #2 \right]}
\def\P{\mathbf{P}}
\newcommand{\prob}[1]{\mathbf{P}\left(#1\right)}
\newcommand{\probcond}[2]{\mathbf{P}\left(\left. #1 \,\right|\, #2 \right)}

\newcommand{\var}[1]{\mathrm{Var}\left[#1\right]}
\newcommand{\cov}[2]{\mathrm{Cov}\left[#1,#2\right]}

\def\cL{\mathcal{L}}
\def\cF{\mathcal{F}}
\newcommand{\One}[1]{\mathbf{1}_{#1}}
\newcommand{\dtv}[2]{\mathrm{d_{TV}}\left(#1,#2\right)}
\newcommand{\Unif}[1]{{\rm Unif}\left(#1\right)}
\newcommand{\Geom}[1]{{\rm Geom}\left(#1\right)}
\newcommand{\TGeom}[2]{{\rm TruncGeom}\left(#1,#2\right)}
\newcommand{\Normal}[2]{{\cal N}\left(#1,#2\right)}
\def\id{\mathrm{id}}
\def\S{\mathfrak{S}}
\def\rev{\mathrm{rev}}
\newcommand{\inv}[1]{{\mathrm{inv}}\left(#1\right)}
\newcommand{\pat}[2]{{\mathrm{pat}}\left(#1,#2\right)}
\newcommand{\occ}[2]{{\mathrm{Occ}}\left(#1,#2\right)}
\newcommand{\occinj}[2]{{\mathrm{Occ^{inj}}}\left(#1,#2\right)}
\newcommand{\Mallows}[2]{{\mathrm{Mallows}}\left(#1,#2\right)}
\newcommand{\MallowsBlock}[1]{{\mathrm{MBlock}}\left(#1\right)}
\newcommand{\ContinuousMallows}[1]{\overset{\longrightarrow}{\mathrm{Mallows}}\left(#1\right)}
\newcommand{\ContinuousBlock}[1]{\overset{\longrightarrow}{\mathrm{MBlock}}\left(#1\right)}
\newcommand{\dens}[2]{{\mathrm{dens}}\left(#1,#2\right)}

\title{Classical patterns in Mallows permutations}
\author{Victor Dubach}
\date{}

\begin{document}

\maketitle

\begin{abstract}
    We study classical pattern counts in Mallows random permutations with parameters $(n,q_n)$, as $n\to\infty$. 
    We focus on three different regimes for the parameter $q = q_n$.
    When $n^{3/2}(1-q)\to0$, we use coupling techniques to prove that pattern counts in Mallows random permutations satisfy a central limit theorem with the same asymptotic mean and variance as in uniformly random permutations.
    When $q\to1$ and $n(1-q)\to\infty$, we use results on the displacements of permutation points to find the order of magnitude of pattern counts.
    When $q\in(0,1)$ is fixed, we use the regenerative property of the Mallows distribution to compare pattern counts with certain $U$-statistics, and establish central limit theorems.
    We also construct a specific Mallows process, that is a coupling of Mallows distributions with $q$ ranging from $0$ to $1$, for which the process of pattern counts satisfies a functional central limit theorem.
\end{abstract}

\section{Background}

\subsection{Mallows random permutations}

Fix a real parameter $q\ge0$ and a positive integer $n\in\N^*$.
We say that a random permutation $\tau_n$ follows the {\em Mallows distribution} with parameter $q$ on the symmetric group $\S_n$, denoted $\tau_n \sim \Mallows{n}{q}$, if:
\[
\text{for all $\sigma\in\S_n$},\quad
\prob{\tau_n = \sigma} = \frac{q^{\inv{\sigma}}}{Z_{n,q}}
\]
where $\inv{\sigma} := \card{ (i,j)\in[n]^2 \,:\, i<j \,,\, \sigma(i)>\sigma(j) }$ is the number of {\em inversions} of $\sigma$, and $Z_{n,q} := \prod_{k=1}^n \frac{1-q^k}{1-q}$ is a normalizing constant.
The parameter $q$ serves as a \enquote{bias} on the number of inversions of $\tau_n$. 
The number of inversions of $\tau_n$ grows stochastically as $q$ grows:
if $q=0$ then $\tau_n = \id_n := 1\,2 \dots n$ almost surely, if $q=1$ then $\tau_n \sim \Unif{\S_n}$ is uniformly random, and if $q\to\infty$ then $\tau_n \to \id_n^{\rev} := n \dots 2\,1$ in probability.

The Mallows distribution was introduced in \cite{M57}, motivated by ranking problems in statistics.
There actually exists a wide range of Mallows-type models \cite{M57,FV86,Mukherjee_2016,zhong2021mallows} that use various other statistics instead of the number of inversions.
The present model is related to several topics, such as $q$-exchangeability \cite{GO10,Gnedin_Olshanski_2012}, the Bruhat decomposition \cite{Diaconis_Simper_2022}, and the asymmetric simple exclusion process \cite{bufetov2022cutoff,Borodin_Bufetov_2024}.
A number of properties of Mallows random permutations have been studied extensively by many authors:
this includes mixing times of related Markov chains \cite{DR00,benjamini2005mixing}, scaling limits \cite{S09,SW18}, longest monotone subsequences \cite{MS13,BP15,BB17}, the degree process \cite{Bhattacharya_Mukherjee_2017}, the cycle structure \cite{GP18,HMV23}, the secreatary problem \cite{Jones_2020,Pinsky_2022_secretary}, binary search trees \cite{ABC21,C23}, logical limit laws \cite{MSV23}, and more.
When studying the asymptotics of such properties as $n\to\infty$, interesting phase transition phenomena often appear when $q$ is allowed to depend on $n$.

A common question when studying random permutations is that of pattern occurrences and avoidance:
given a fixed permutation $\pi\in\S_r$ and a random permutation $\tau\in\S_n$, what is the distribution of the number of occurrences of $\pi$ as a (classical, consecutive, or vincular) pattern in $\tau$, and what is the probability that $\tau$ avoids it?
This question has been partially addressed for Mallows random permutations.
Central limit theorems have been established for the number of inversions \cite{DR00,Diaconis_Simper_2022}, for the number of descents \cite{borodin2010adding,H22}, and for the number of consecutive pattern occurrences \cite{CDE18}.
Growth rates for the probability of avoiding a consecutive pattern \cite{CDE18} or a classical pattern of length three \cite{P21} have been studied, and thresholds for the appearance of a given pattern have been established in \cite{Bevan_Threlfall_2024}\footnote{
The paper \cite{Bevan_Threlfall_2024} is about random permutations with a given number of inversions rather than Mallows random permutations, but using concentration results on the number of inversions in Mallows random permutations (see e.g.~\cite[Proposition~29]{Bevan_Threlfall_2024}, and \Cref{eq: inversions as sum of independent} here), the results of \cite{Bevan_Threlfall_2024} can be transferred from their model to ours.
}.
However, except for inversions, there seems to be a lack of results concerning the number of occurrences of classical patterns.
This paper initiates this study for three different regimes of $q=q_n$.

Note that these questions have well-studied analogs in the field of random graphs.
Indeed, Mallows random permutations are analogous to Erd\H{o}s--Rényi random graphs (the permutation graph of a Mallows permutation is an Erd\H{o}s--Rényi graph conditioned to be a permutation graph), and pattern occurrences then correspond to subgraph occurrences.
Subgraph counts and subgraph avoidance in Erd\H{o}s--Rényi random graphs have been studied extensively (see e.g.~\cite{janson2011random} and references therein); this further motivates the questions raised here.

\subsection{Previous results on pattern counts}

Let $\tau\in\S_n$ be a permutation and $I = \{i_1< \dots< i_r\} \subseteq [n]$ be a subset of $r$ indices.
The {\em pattern} induced by $\tau$ on $I$, denoted $\pat{I}{\tau}$, is the permutation $\sigma \in \S_r$ defined by:
\[
\text{for all } j,k \in [r],\quad
\sigma(j) < \sigma(k) \Longleftrightarrow \tau\left(i_j\right) < \tau\left(i_k\right) .
\]
Now fix $\pi\in\S_r$.
A {\em classical occurrence} of $\pi$ as a pattern in $\tau$ is a subset $I\subseteq[n]$ such that $\pat{I}{\tau} = \pi$.
For instance, a classical occurrence of the pattern $\id_r := 1\,2\dots r$ is an increasing subsequence of length $r$, and a classical occurrence of the pattern $2\,1$ is an inversion.
The number of classical occurrences of $\pi$ as a pattern in $\tau$ is then:
\[
\occ{\pi}{\tau} := \card{ I\subseteq[n] \,:\, \pat{I}{\tau} = \pi } .
\]
In this paper we are only interested in these classical (not consecutive, nor vincular) pattern occurrences, so we will simply call them pattern occurrences.
They have notable applications to independence testing \cite{Hoeffding_1994} and parameter testing \cite{Hoppen_Kohayakawa_Moreira_Sampaio_2011}.

The question tackled in this paper is that of the asymptotics of $\occ{\pi}{\tau_n}$ as $n\to\infty$, when $\pi\in\S_r$ is fixed and $\tau_n \sim \Mallows{n}{q_n}$.
In this framework, we can assume w.l.o.g.~that $0 < q_n \le 1$. 
Indeed, the reversed permutation $\tau_n^{\rev} := \tau_n(n) \dots \tau_n(2) \, \tau_n(1)$ follows the $\Mallows{n}{q_n^{-1}}$ distribution, and $\occ{\pi}{\tau_n} = \occ{\pi^\rev}{\tau_n^\rev}$ for all patterns $\pi$.
Before stating our results, let us review the state of the art on this question.

\subsubsection{The uniform case}

Since $\Mallows{n}{1} = \Unif{\S_n}$, it is instructive to consider the simpler case of uniformly random permutations before studying general Mallows permutations.
In this case, the asymptotics of $\occ{\pi}{\tau_n}$ are well known:

\begin{theorem}[\cite{JNZ15}]\label{th: CLT occ unif}
    Fix $r\ge2$ and $\pi\in\S_r$.
    For each $n$, let $\tau_n \sim \Unif{\S_n}$.
    Then we have the following convergence in distribution:
    \begin{equation*}
        \frac{\occ{\pi}{\tau_n} - \frac{1}{r!}\binom{n}{r}}{n^{r-1/2}}
        \cv{n\to\infty} \Normal{0}{\nu_\pi^2} ,
    \end{equation*}
    for some explicit $\nu_\pi>0$, where $\Normal{0}{\nu_\pi^2}$ is the centered normal distribution with variance $\nu_\pi^2$.
    This also holds with convergence of all moments, and jointly for any finite family of patterns (with explicit covariances).
\end{theorem}

Many more results, such as concentration inequalities and bounds on the speed of convergence, are available in the uniform case.
However, the focus of this paper is on the first- and second-order asymptotics of $\occ{\pi}{\tau_n}$.

\subsubsection{First-order convergence through permuton theory}

The first-order asymptotics of pattern counts are closely related to the theory of {\em permutons}, which are certain probability measures on $[0,1]^2$ that serve as scaling limits for large permutations \cite{HKMRMS13,Glebov_Grzesik_Klimosova_Kral_2015}.
Indeed, for each integer $n$, let $\tau_n\in\S_n$ be a random permutation.
Then $\tau_n$ converges in probability as $n\to\infty$ to a limit permuton $\mu$ if and only if for all $r\in\N^*$ and all $\pi\in\S_r$, the random variables $\binom{n}{r}^{-1} \occ{\pi}{\tau_n}$ converge in probability as $n\to\infty$ to the quantity 
\[
\dens{\pi}{\mu} := \int_{\left([0,1]^2\right)^r} \sum_{\sigma,\rho \in\S_r} \One{x_{\sigma(1)}< \dots< x_{\sigma(r)}} \One{y_{\rho(1)}< \dots< y_{\rho(r)}} \One{\rho^{-1}\circ\sigma =\pi} \, \mu(dx_1,dy_1) \dots \mu(dx_r,dy_r) ,
\]
see \cite{HKMRMS13} for details.
In general, it may happen that $\dens{\pi}{\mu} = 0$; the correct order of magnitude of $\occ{\pi}{\tau_n}$ should then be determined by different methods.

The permuton limits of Mallows random permutations are well understood.
Suppose that $n(1-q_n) \cv{n\to\infty} \beta$ for some $\beta \in [0,\infty]$.
Then:
\begin{itemize}
    \item If $\beta < \infty$ then $\tau_n$ converges in probability to an explicit deterministic permuton $\mu_\beta$ \cite{S09} with a positive density. 
    Thus
    \begin{equation}\label{eq: first order in thermodynamic regime}
        \binom{n}{r}^{-1} \occ{\pi}{\tau_n} \cv{n\to\infty} \dens{\pi}{\mu_\beta} >0
    \end{equation}
    in probability, for all patterns $\pi$.
    If $\beta=0$ then $\mu_\beta$ is the Lebesgue measure on $[0,1]^2$ and $\dens{\pi}{\mu_\beta} = \frac{1}{r!}$: 
    in this regime, the first-order asymptotics of $\occ{\pi}{\tau_n}$ are the same as in the uniform case.
    \item If $\beta=\infty$, the permuton limit of $\tau_n$ can be found with \cite[Theorem~1.1]{BP15}.
    Indeed, for any $\epsilon>0$:
    \[
    \text{for all } i\in[n],\quad
    \prob{\abs{\tau_n(i)-i} \ge \epsilon n} \le 2q_n^{\epsilon n}
    \]
    where $q_n^{\epsilon n} \to 0$ as $n\to\infty$, thus $\frac1n \card{ i\in[n] : \abs{\tau_n(i)-i} \ge \epsilon n }$ converges in probability to $0$.
    As a consequence, $\tau_n$ converges in probability to the diagonal permuton $\mu_\Delta$, which puts uniform mass on the diagonal of the unit square.
    Therefore:
    \begin{equation}\label{eq: nr first-order for Mallows patterns}
        \binom{n}{r}^{-1} \occ{\pi}{\tau_n} \cv{}
        \left\{
        \begin{array}{ll}
        0 & \text{ if } \pi \ne \id_r;\\
        1 & \text{ if } \pi = \id_r;
        \end{array}
        \right.
    \end{equation}
    in probability, as $n\to\infty$.
    When $\pi\ne\id_r$, this does not provide the correct order of magnitude for $\occ{\pi}{\tau_n}$.
\end{itemize}

\subsubsection{Asymtotic normality of inversions}

To our knowledge, the only pattern for which more is known is the inversion pattern $\pi=2\,1$.
First-order asymptotics for the number of inversions, in many regimes of $q=q_n$, are given in \cite{Pinsky_2022_inversions}.
As remarked in \cite{DR00}, a central limit theorem can also be proved; we state it in a restricted case for simplicity, and provide a proof in \hyperref[sec: appendix 1]{Appendix} for completeness.

\begin{theorem}\label{th: inv CLT if above thermo}
    For each $n$, let $\tau_n \sim \Mallows{n}{q_n}$.
    Suppose that $n(1-q_n)\to0$ as $n\to\infty$.
    Then
    \begin{equation*}
        \expec{\inv{\tau_n}} = \tfrac14n^2 - \tfrac{1}{36}(1-q_n)n^3 + \Landau{o}{}{n^{3/2} + (1-q_n)n^3}
        \quad\text{and}\quad
        \var{\inv{\tau_n}} = \tfrac{1}{36}n^3 + \Landau{o}{}{n^3} .
    \end{equation*}
    Furthermore, we have the following convergence in distribution:
    \begin{equation*}
        \frac{\inv{\tau_n} - \expec{\inv{\tau_n}}}{\var{\inv{\tau_n}}^{1/2}}
        \cv{n\to\infty} \Normal{0}{1}.
    \end{equation*}
\end{theorem}

Note that when $q_n=1$, we retrieve the famous central limit theorem for the number of inversions in a uniformly random permutation:
\begin{equation}\label{eq: inv CLT if uniform}
    \frac{\inv{\tau_n} - \frac14 n^2}{\frac16 n^{3/2}} \cv{n\to\infty} \Normal{0}{1}.    
\end{equation}

\section{Overview of our results}

The aim of this paper is to broaden the asymptotic study of $\occ{\pi}{\tau_n}$ for $\pi\ne2\,1$ and $\tau_n \sim \Mallows{n}{q}$, in three different regimes of $q=q_n$.
Our results are summarized in the following sections:
we present in \Cref{sec: results almost uniform} a central limit theorem when $n^{3/2}(1-q)\to0$, in \Cref{sec: results transition} first-order asymptotics when $q\to1$ and $n(1-q)\to0$, and in \Cref{sec: results regenerative} joint central limit theorems when $q$ does not depend on $n$.

\subsection{Notation}\label{sec: notation}

Let us introduce some notation beforehand.
Throughout this paper, the parameter $q\in(0,1)$ may depend on $n$ (this is the case in \Cref{sec: results almost uniform,sec: results transition}), but we often write $q$ instead of $q_n$ for simplicity.

We let $\N := \{0,1,2,\dots\}$ be the set of nonnegative integers, $\N^* := \{1,2,3,\dots\}$ be the set of positive integers, $[n] := \{1,2,\dots,n\}$, and $\S_n$ be the symmetric group on $[n]$.
If $\sigma\in\S_n$, the size of $\sigma$ is the integer $n$ and denoted by $\abs{\sigma}$. 
We let $\S := \bigcup_{n\in\N^*} \S_n$ be the set of all finite permutations, and $\S_\infty$ be the symmetric group on $\N^*$.

All permutations $\sigma\in\S_\ell$ in this paper are written in one-line notation $\sigma = \sigma(1)\, \sigma(2)\dots \sigma(\ell)$. 
The {\em direct sum} of $\sigma\in\S_\ell$ and $\rho\in\S_m$ is the permutation $\sigma\oplus\rho \in \S_{\ell+m}$ such that
\[
\sigma\oplus\rho
= \sigma(1)\; \sigma(2)\, \dots\, \sigma(\ell)\;
\rho(1)\p\ell\; \rho(2)\p\ell\, \dots\, \rho(m)\p\ell .
\]
Note that the $\oplus$ operator on $\S$ is associative.
When a permutation $\sigma\in\S$ can not be written as the direct sum of two non-empty permutations, we say that it is {\em indecomposable}, or that it is a {\em block}.
For example, the only indecomposable permutations in $\S_3$ are $2\,3\,1$, $3\,1\,2$, and $3\,2\,1$.
Conversely, $\sigma$ is decomposable if there exists $k\in[n-1]$ such that $\{ \sigma(1), \dots, \sigma(k) \} = [k]$.
It is a classical fact that any permutation $\sigma\in\S$ can be uniquely decomposed as a sum of blocks $\sigma = \sigma_1\oplus \dots\oplus \sigma_d$,
and we call the integer $d\in\N^*$ the {\em number of blocks} of $\sigma$.

If $(E_n)_{n\in\N}$ is a sequence of events, we say that $E_n$ happens with high probability (w.h.p.) when $\prob{E_n}\to1$ as $n\to\infty$.
If $(X_n)_{n\in\N}$ and $(Y_n)_{n\in\N}$ are sequences of positive random variables, we write $X_n = \Landau{\O_\P}{}{Y_n}$ when:
\[
\sup_{n\in\N}\prob{ X_n \ge C\, Y_n } \cv{C\to+\infty} 0 ,
\]
or equivalently when $\left( X_n / Y_n \right)_{n\in\N}$ is tight.
Also, $X_n = \Landau{o_\P}{}{Y_n}$ means that $X_n / Y_n \to0$ in probability as $n\to\infty$.

In this paper, we use the notation $x_n = \landauT{\Theta}{}{y_n}$ when there exist constants $\alpha, \beta, \gamma > 0$ and $n_0\in\N$ such that:
\begin{equation*}
    \text{for all }n\ge n_0,\quad
    \alpha\, y_n \le x_n \le \beta\, (\log n)^\gamma\, y_n .
\end{equation*}
Note that this slightly differs from the usual ${\widetilde\Theta}$ notation: here, there is no logarithmic factor in the lower bound.
Next, if $X_n$, $Y_n$ are real r.v.'s, we write $X_n = \landauT{\Theta}{\P}{Y_n}$ when there exist constants $\beta, \gamma > 0$ such that for any $\epsilon>0$, there exist $\alpha_\epsilon>0$ and $n_\epsilon\in\N$ such that:
\begin{equation*}
    \text{for all }n\ge n_\epsilon,\quad
    \prob{ \alpha_\epsilon\, Y_n \le X_n \le \beta\, (\log n)^{\gamma}\, Y_n } \ge 1-\epsilon .
\end{equation*}
Informally, this means that $X_n$ has the same asymptotic order as $Y_n$ in probability, but the upper bound suffers a logarithmic factor.

\subsection{Stability close to the uniform regime}\label{sec: results almost uniform}

One might expect the $\Mallows{n}{q}$ distribution to have similar properties to $\Unif{\S_n}$ when $q$ is close enough to~$1$.
For instance, by \Cref{th: inv CLT if above thermo}, the convergence \eqref{eq: inv CLT if uniform} holds for $\tau_n\sim \Mallows{n}{q}$ if $n^{3/2}(1-q)\to0$ (otherwise $\expec{\inv{\tau_n}}$ has an additional term above $n^{3/2}$).
The same stability phenomenon occurs for all patterns, as can be deduced from a coupling lemma:

\begin{lemma}\label{lem: coupling Mallows with uniform}
    Let $n\in\N^*$ and $q\in(0,1)$ such that $n(1-q)\le1$.
    We can couple $\tau_n \sim \Mallows{n}{q}$ with $u_n \sim \Unif{\S_n}$ in such a way that there exists $\cH_n \subset [n]$ satisfying $\pat{[n]\setminus\cH_n}{\tau_n} = \pat{[n]\setminus\cH_n}{u_n}$ and $\expec{\cH_n} \le \frac32 n^2 (1-q)$.
\end{lemma}

\begin{corollary}\label{th: stability close to uniform}
    Fix $r\ge2$ and $\pi\in\S_r$.
    For each $n$ let $\tau_n \sim \Mallows{n}{q}$, where $n^{3/2}(1-q) \to 0$ as $n\to\infty$.
    Then we have the following convergence in distribution:
    \begin{equation*}
        \frac{\occ{\pi}{\tau_n} - \frac{1}{r!}\binom{n}{r}}{n^{r-1/2}}
        \cv{n\to\infty} \Normal{0}{\nu_\pi^2} ,
    \end{equation*}
    for the same $\nu_\pi>0$ as in \Cref{th: CLT occ unif}.
    This also holds jointly for any finite family of patterns, with the same covariances as in \Cref{th: CLT occ unif}.
\end{corollary}

This is proved in \Cref{sec: proofs almost uniform}.
There are several ways to prove \Cref{lem: coupling Mallows with uniform};
for example, we could use a point process approach and \cite[Lemma~4.2]{MS13}.
Here, however, we rather opt for a coupling via inversion sequences, as it is more in line with the rest of the paper and relies on different ideas.
We believe that our approach is interesting on its own, and that appropriate variants of it could be used to study other statistics of Mallows permutations.

It starts from the construction of Mallows permutations using independent truncated geometric variables with parameters $(k,q)$ (see \Cref{sec: preliminaries}).

If $n^2(1-q)\to0$, then it is possible to couple the $\TGeom{k}{q}$ variables with $\Unif{[k]}$ ones so that they are all equal with high probability (see the computation in \cite[Lemmas~7.2 and~7.3]{MSV23}, and \cite[Remark~2]{BB17}).
In this regime, the permutation $\tau_n \sim \Mallows{n}{q}$ can be replaced by a uniformly random permutation $u_n \sim \Unif{\S_n}$ with high probability, and the convergence of \Cref{th: stability close to uniform} is straightforward.

In our regime, that is when $n^{3/2}(1-q)\to0$, such a strong coupling is no longer available.
A naive idea would be to couple $\TGeom{k}{q}$ with $\Unif{[k]}$ so that as many variables as possible are equal, but this would not translate well into pattern occurrences.
The coupling of \Cref{lem: coupling Mallows with uniform} is more subtle: 
when some $\TGeom{k}{q}$ variable does not equal the corresponding $\Unif{[k]}$ variable, the subsequent $\TGeom{k'}{q}$ variables should rather try to be equal to an appropriate \enquote{local shift} of the corresponding $\Unif{[k']}$ variables.
This shift is chosen so that $\tau_n$ and $u_n$ have a large pattern in common, with high probability.
See \Cref{sec: proofs almost uniform} for more details.

\subsection{First-order asymptotics in the transition regime}\label{sec: results transition}

Now suppose that $q\to1$ and $n(1-q)\to\infty$ as $n\to\infty$.
In this regime, \eqref{eq: nr first-order for Mallows patterns} only tells us that $\occ{\pi}{\tau_n} = \Landau{o_\P}{}{n^r}$  for $\pi\ne\id_r$.
The following theorem describes the correct first-order asymptotics of $\occ{\pi}{\tau_n}$, up to logarithmic factors.

\begin{theorem}\label{th: order of occ for transition regime}
    Let $q = q_n \in (0,1)$ such that $q \to 1$ and $n(1-q) \to \infty$ as $n\to\infty$.
    Fix $\pi\in\S_r$, written as a sum of blocks $\pi = \pi_1 \oplus \dots \oplus \pi_d$.
    Then, if $\tau_n \sim \Mallows{n}{q}$:
    \begin{equation*}
        \occ{\pi}{\tau_n} = \landauT{\Theta}{\P}{ \frac{n^d}{(1-q)^{r-d}} } .
    \end{equation*}
    Moreover, for all $p>0$:
    \begin{equation*}
        \expec{ \occ{\pi}{\tau_n}^p } = \landauT{\Theta}{}{ \left( \frac{n^d}{(1-q)^{r-d}} \right)^p } .
    \end{equation*}
\end{theorem}

In particular if $q = 1 - c n^{-x}$ for some $c>0$ and $x\in(0,1)$, then $\occ{\pi}{\tau_n}$ has order $n^{xr + (1-x)d}$.
This shows how the first-order asymptotics of $\occ{\pi}{\tau_n}$ interpolate between $n^d$ for fixed $q$ (see \Cref{sec: results regenerative}) and $n^r$ for bounded $n(1-q)$ (recall \Cref{eq: first order in thermodynamic regime}).
For $\pi = 2\,1$ (then $r=2$ and $d=1$), this also coincides with \cite[Theorem~1.3]{Pinsky_2022_inversions}.

\Cref{th: order of occ for transition regime} is proved in \Cref{sec: proofs transition}.
The idea behind the proof is that, while $\tau_n \sim\Mallows{n}{q}$ is in general indecomposable if $q\to1$, it can be approximated by a decomposable permutation after removing a fraction of points if $n(1-q)\to\infty$.
This allows us to approximate $\tau_n$ by a sum of subpermutations $\tau_{n,1} \oplus\dots\oplus \tau_{n,m}$, and it remains to estimate the number of pattern occurrences in these subpermutations.
The upper and lower bounds are then proved separately: 
the former hinges on a trivial upper bound for the number of pattern occurrences, while the latter makes use of \eqref{eq: first order in thermodynamic regime} by finding smaller subpermutations in $\tau_n$ which are $\Mallows{n_k}{q}$ distributed with bounded $n_k(1-q)$.

\subsection{Joint and process convergence in the regenerative regime}\label{sec: results regenerative}

In this section, $q\in(0,1)$ does not depend on $n$.
This regime is called \enquote{regenerative} for Mallows random permutations, because $\tau_n \sim \Mallows{n}{q}$ can be decomposed into a truncated sum of i.i.d.~blocks.
This property comes from the embedding of finite Mallows permutations into the {\em infinite Mallows permutation} $\Sigma \sim \Mallows{\N^*}{q}$.
This permutation, introduced in \cite[Definition~4.4]{GO10} and recalled in \Cref{sec: infinite Mallows q fixed}, is a random element of $\S_\infty$ with two notable properties.
The first one is that $\pat{[n]}{\Sigma}$, the pattern induced by $\Sigma$ on $[n]$, follows the $\Mallows{n}{q}$ distribution (\Cref{lem: pattern of infinite Mallows}).
The second one is that $\Sigma$ is regenerative in the sense that
\[
\Sigma = B_1\oplus \dots\oplus B_k\oplus \dots
\]
where $\left(B_k\right)_{k\in\N^*}$ is an infinite sequence of i.i.d.~random finite blocks (\Cref{lem: infinite Mallows is regenerative}).
We say that $B_k \in \S$ is a {\em Mallows block with parameter $q$}, and we denote by $\MallowsBlock{q}$ its distribution.

The $\MallowsBlock{q}$ distribution is essential when describing the asymptotics of $\occ{\pi}{\tau_n}$.
If $\pi\in\S$ is written as a sum of blocks $\pi=\pi_1\oplus \dots\oplus\pi_d$, define
\begin{equation}
    e_{\pi,q} := \prod_{j=1}^d \frac{\expec{\occ{\pi_j}{B}}}{\expec{\abs{B}}}
\end{equation}
where $B \sim \MallowsBlock{q}$.
Another equivalent expression of $e_{\pi,q}$ will be given later in \Cref{sec: proof time-joint CLT}, \Cref{eq: double expression for e}.
Our results show the expansion 
\[
\occ{\pi}{\tau_n} = \binom{n}{d}e_{\pi,q} + n^{d-1/2} \Normal{0}{\gamma_\pi^2} + \Landau{o_\P}{}{n^{d-1/2}}
\]
for some $\gamma_\pi^2\ge0$, where $d\in\N^*$ is the number of blocks of $\pi\in\S$.
When $\pi$ is not increasing, the Gaussian fluctuations are non-degenerate in the sense that $\gamma_\pi^2>0$ (\Cref{prop: non-degeneracy of CLT and rank of matrix}).
When $\pi$ is increasing, non-degenerate Gaussian fluctuations can be found at a lesser order (\Cref{th: CLT for identity pattern with fixed q}).
We prove two multivariate versions of this central limit theorem:
the first one, \Cref{th: LLN and CLT for joint patterns with fixed q}, holds jointly in all patterns having $d$ blocks;
the second one, \Cref{th: Donsker for time-joint single pattern}, holds jointly in $q\in(0,1)$ for a certain continuous-time coupling of the $\Mallows{n}{q}$ distributions.

In what follows, if $d\ge1$ is fixed, we write \enquote{$\pi= \pi_1\oplus\dots\oplus\pi_d$} to mean that $\pi$ runs over the set of patterns in $\S$ having exactly $d$ blocks, without size constraints.

\begin{theorem}\label{th: LLN and CLT for joint patterns with fixed q}
    Fix $d\ge1$ and $q\in(0,1)$.
    Let $\Sigma\sim \Mallows{\N^*}{q}$, and define $\tau_n := \pat{[n]}{\Sigma} \sim\Mallows{n}{q}$ for each $n\ge1$.
    For any $\pi = \pi_1\oplus\dots\oplus\pi_d$, we have the following almost sure convergence:
    \begin{equation}\label{eq: LLN q fixed}
        \frac{1}{n^d} \occ{\pi}{\tau_n} 
        \,\overset{}{\longrightarrow}\,
        \frac{1}{d!} e_{\pi,q} \,.
    \end{equation}
    Moreover, jointly in all patterns having $d$ blocks, we have the following finite-dimensional convergence in distribution as $n\to\infty$:
    \begin{equation}\label{eq: CLT joint q fixed}
        \left(
        \frac{ \occ{\pi}{\tau_n} - \binom{n}{d} e_{\pi,q} }{n^{d-1/2}} 
        \right)_{\pi=\pi_1\oplus\dots\oplus\pi_d}
        \,\cv{}\, \Normal{0}{\Gamma^{(d)}}
    \end{equation}
    for some infinite covariance matrix $\Gamma^{(d)}$ described by \eqref{eq: formula covariance matrix q fixed}.
    This also holds with convergence of all joint moments.
\end{theorem}

In order to understand the (non-)degeneracy of the limit variables in \eqref{eq: CLT joint q fixed}, as well as the asymptotic correlation between pattern counts, we further study the matrix $\Gamma^{(d)}$.

\begin{proposition}\label{prop: non-degeneracy of CLT and rank of matrix}
    Fix $d\ge1$, $q\in(0,1)$, and let $\Gamma^{(d)}$ be the infinite matrix of \Cref{th: LLN and CLT for joint patterns with fixed q}.
    Then for any fixed $\pi=\pi_1 \oplus\dots\oplus \pi_d$ which is not an identity permutation, the central limit theorem for $\occ{\pi}{\tau_n}$ is non-degenerate in the sense that $\Gamma^{(d)}_{\pi,\pi} >0$.
    On the other hand, $\Gamma^{(d)}_{\id_d,\id_d} =0$.
    Furthermore for any $r\ge d$, if we denote by $\Gamma^{(d,r)}$ the restriction of $\Gamma^{(d)}$ to patterns in $\S_r$:
    \begin{itemize}
        \item if $d=r$, then $\Gamma^{(r,r)} = (0)$;
        \item if $d=r-1$, then $\Gamma^{(d,r)}$ is invertible;
        \item if $d=r-2$, then $\Gamma^{(d,r)}$ has rank $7$ if $d=2$, and rank $4d$ if $d\ge3$.
        Its size is $\binom{d}{2}+3d$, thus it is invertible if and only if $d\le3$;
        \item if $d=1$, then $\Gamma^{(1,r)}$ is invertible.
    \end{itemize}
\end{proposition}

In particular, note that the matrix $\Gamma^{(d,r)}$ is not always invertible.
The fact that $\Gamma^{(d)}_{\id,\id} =0$ implies that the convergence \eqref{eq: CLT joint q fixed} is degenerate for $\pi=\id$.
The fluctuations of $\occ{\id}{\tau_n}$ can actually be found at a lower order:

\begin{theorem}\label{th: CLT for identity pattern with fixed q}
    Fix $r\ge2$, $q\in(0,1)$, and let $\tau_n \sim\Mallows{n}{q}$ for each $n\ge1$.
    We have the following convergence in distribution as $n\to\infty$:
    \begin{equation*}
        \frac{ \occ{\id_r}{\tau_n} - \binom{n}{r} + \binom{n}{r-1} (r-1) \frac{\expec{\inv{B}}}{\expec{\abs{B}}} }{n^{r-3/2}} 
        \,\cv{}\, \Normal{0}{\gamma_r^2}
    \end{equation*}
    where $B\sim\MallowsBlock{q}$, and $\gamma_r^2 > 0$ is given by \eqref{eq: asymptotic variance increasing subsequences}.
    This also holds with convergence of all moments.
\end{theorem}

The previous three results are proved in \Cref{sec: proofs pattern-joint CLT q fixed}.
They heavily rely on the regenerative property of Mallows random permutations, which allows us to approximate
\begin{equation}\label{eq: approx occ U-stat}
    \occ{\pi}{\tau_n}
    \approx \sum_{1\le k_1<\dots<k_d\le K_n} \prod_{j=1}^d \occ{\pi_j}{B_{k_j}} ,
\end{equation}
where $K_n$ is an appropriate stopping time.
One may recognize here a type of $U$-statistic; the proofs then follow with the same techniques as in \cite{J18,J23_forest}.

\medskip

A recent line of work is the introduction of {\em continuous-time Mallows processes} in \cite{C22}.
There, a (continuous-time) Mallows process is defined as an $\S_n$-valued càdlàg stochastic process $\left( \tau_{n,t} \right)_{t\in[0,1)}$ with Mallows marginals, i.e.~$\tau_{n,t} \sim \Mallows{n}{t}$ for all $t\in[0,1)$.
There are several properties that a Mallows process may or may not have; in particular, the author of \cite{C22} studies what they call {\em smooth} and {\em regular} Mallows process.
Here we propose another natural approach to define a Mallows process, which turns out to be smooth but {\it a priori} not regular.

In \Cref{sec: CTI Mallows}, we define the {\em continuous-time infinite Mallows process} $\left( \Sigma_t \right)_{t\in[0,1)} \sim \ContinuousMallows{\N^*}$, which is an infinite analog of the regular Mallows process in \cite{C22}, and a \enquote{unilateral} variant of the bi-infinite Mallows process in \cite{AK24}.
We then define $\left( \tau_{n,t} \right)_{t\in[0,1)} := \left(\big. \pat{[n]}{\Sigma_t} \right)_{t\in[0,1)}$, for each $n\in\N^*$.
It is easy to check that, with the terminology of \cite{C22}, this is a smooth Mallows process for each $n$.

An important property of the continuous-time infinite Mallows process is that for any $q\in(0,1)$, the restricted process $\left( \Sigma_t \right)_{t\in[0,q]}$ can be decomposed into a sum of i.i.d.~permutation processes, called {\em $q$-block processes} (\Cref{th: CTI Mallows process as sum of blocks}).
Thanks to this, $\occ{\pi}{\tau_{n,t}}$ can be approximated by a $U$-statistic as in \eqref{eq: approx occ U-stat}, jointly in $t\in[0,q]$ for any fixed $q\in(0,1)$.
This yields the following process convergence:

\begin{theorem}\label{th: Donsker for time-joint single pattern}
    Let $\left( \Sigma_t \right)_{t\in[0,1)} \sim \ContinuousMallows{\N^*}$, and define $\tau_{n,t} := \pat{[n]}{\Sigma_t} \sim \Mallows{n}{t}$ for all $n\in\N^*$ and $t\in[0,1)$.
    Fix a pattern $\pi\in\S$ with $d$ blocks.
    Then, as $n\to\infty$, the finite-dimensional distributions of the stochastic process
    \[
    t\in[0,1) \mapsto
    \frac{ \occ{\pi}{\tau_{n,t}} - \binom{n}{d} e_{\pi,t} }{n^{d-1/2}} 
    \]
    weakly converge to that of the centered Gaussian process $\left( X_{\pi,t} \right)_{t\in[0,1)}$ with covariance function $H_\pi$ defined by \Cref{eq: formula covariance function Donsker}\footnote{
    The formula given in \eqref{eq: formula covariance function Donsker} for $H_\pi(s,t)$ involves the law of the $q$-block processes, for an arbitrary choice of $q \ge \max(s,t)$.
    Although this law depends on $q$, the value of $H_\pi(s,t)$ does not, since the discrete process $\occ{\pi}{\tau_{n,\cdot}}$ is independent of $q$.
    }.
    This also holds with convergence of all joint moments.
\end{theorem}

Simulations of the random process $\left( \occ{\pi}{\tau_{n,t}} \right)_{t\in[0,1)}$ are shown in \Cref{fig: inversion process} for $\pi = 21$, and in \Cref{fig: occ size3 process} for three other patterns.
The value of $e_{\pi,t}$ is known to be $\frac{t}{1-t}$ for $\pi\in\{21, 132, 213\}$ (see \Cref{rem: explicit formula for e with single inversion}), but is unknown for most other patterns.
Simulations suggest that $e_{\pi,t}$ might be proportional to $\frac{t^2}{(1-t)^2}$ for $\pi\in\{231, 312, 321\}$.

\begin{figure}[ht]
    \centering
    \includegraphics[width=.47\linewidth]{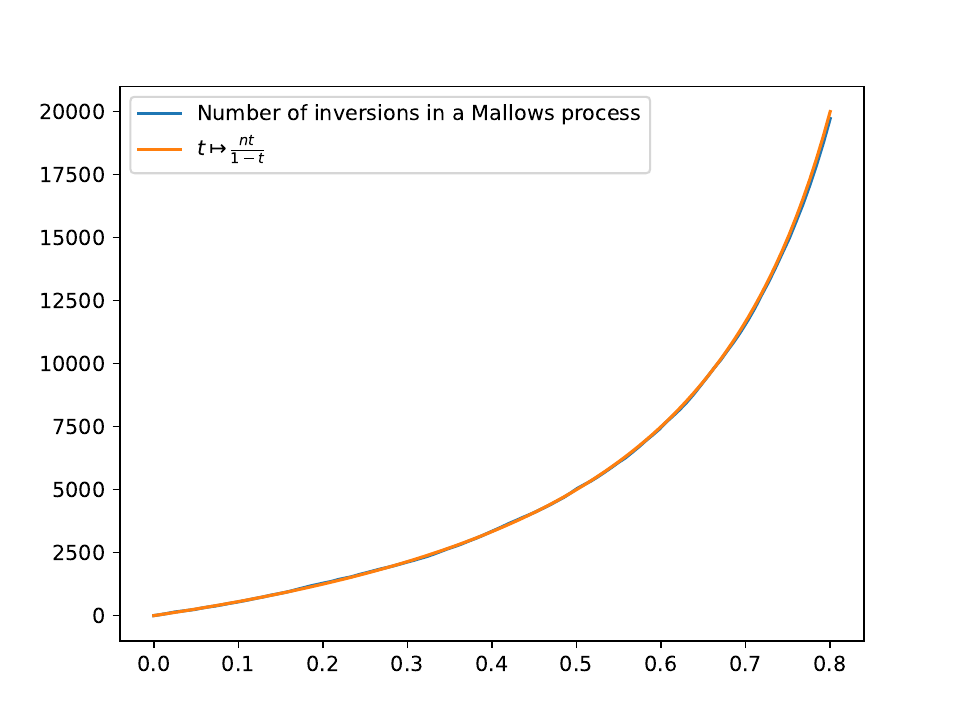}
    \includegraphics[width=.47\linewidth]{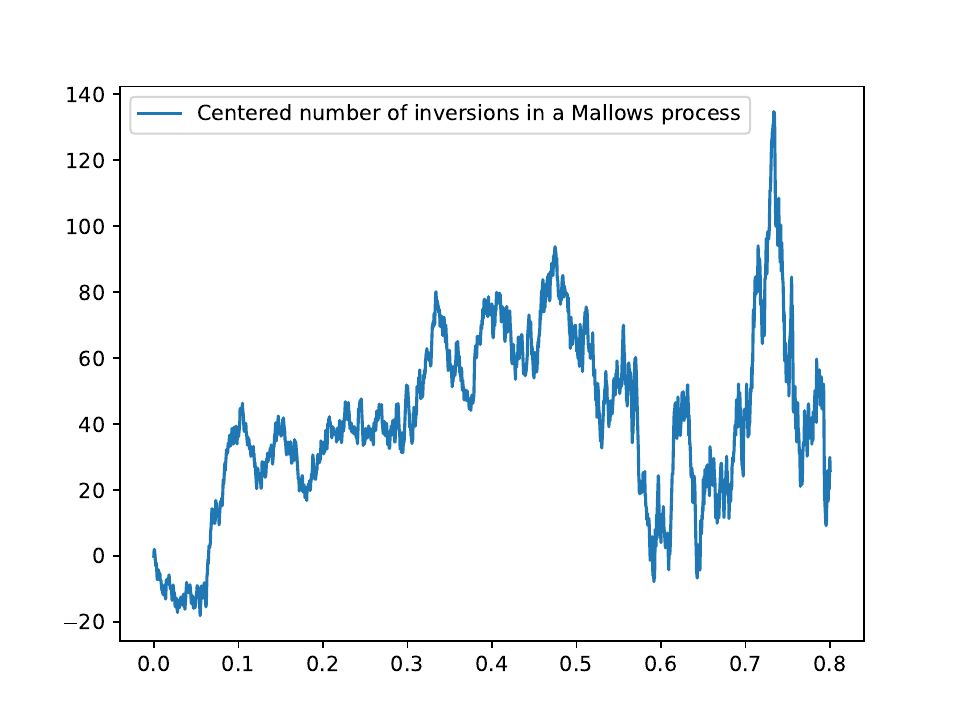}
    \includegraphics[width=.47\linewidth]{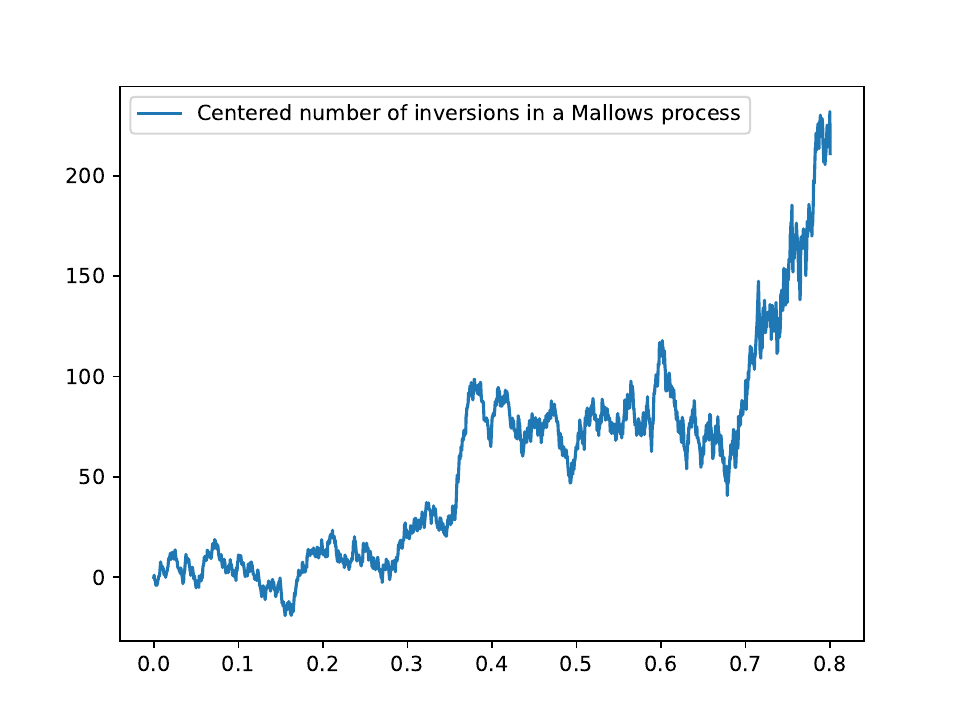}
    \includegraphics[width=.47\linewidth]{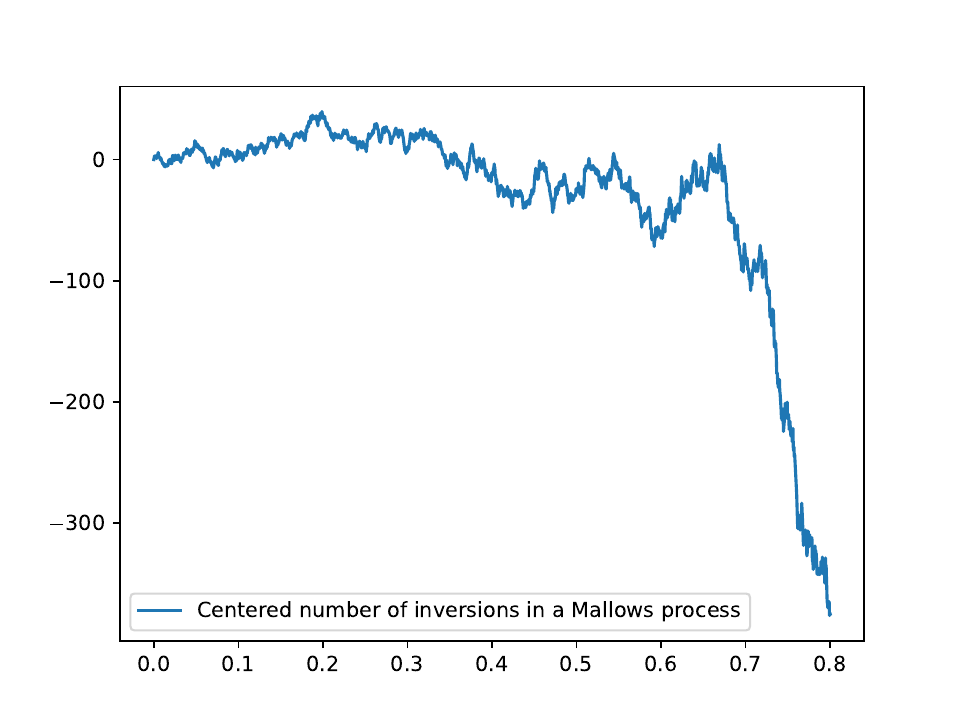}
    \caption{The first figure shows a simulation of $\left( \inv{\tau_{n,t}} \right)_{t\in(0,q]}$ and its first-order limit $\left( n\, e_{2 1,t} = \frac{nt}{1-t} \right)_{t\in(0,q]}$. 
    The other three show simulations of the centered process $\left( \inv{\tau_{n,t}} - n\, e_{2 1,t} \right)_{t\in(0,q]}$, approximating $\sqrt n \cdot \left( X_{2 1,t} \right)_{t\in(0,q]}$.
    This is done with $n=5000$ and up to time $q=0.8$.
    Note that the variance of the process $\left( X_{2 1,t} \right)$ seems to increase as $t$ gets closer to $1$:
    this is unsurprising when taking into account the fact that $\occ{\pi}{\tau_{n,t_n}}$ has a higher order of magnitude when $t_n \cv{n\to\infty} 1$.}
    \label{fig: inversion process}
\end{figure}

\begin{figure}[ht]
    \centering
    \includegraphics[width=.47\linewidth]{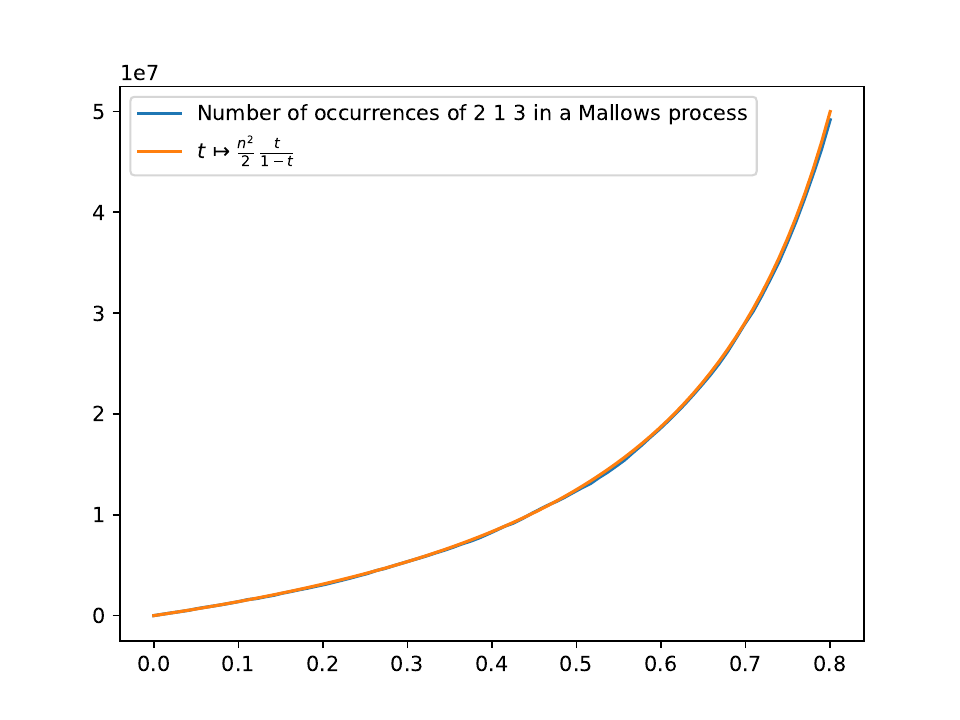}
    \includegraphics[width=.47\linewidth]{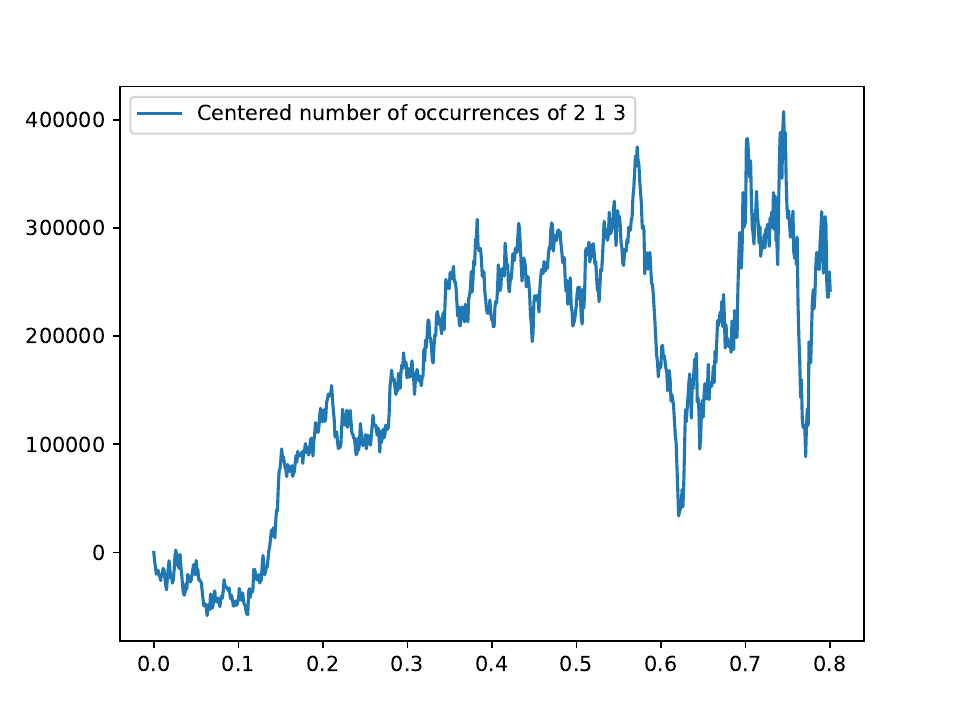}
    \includegraphics[width=.47\linewidth]{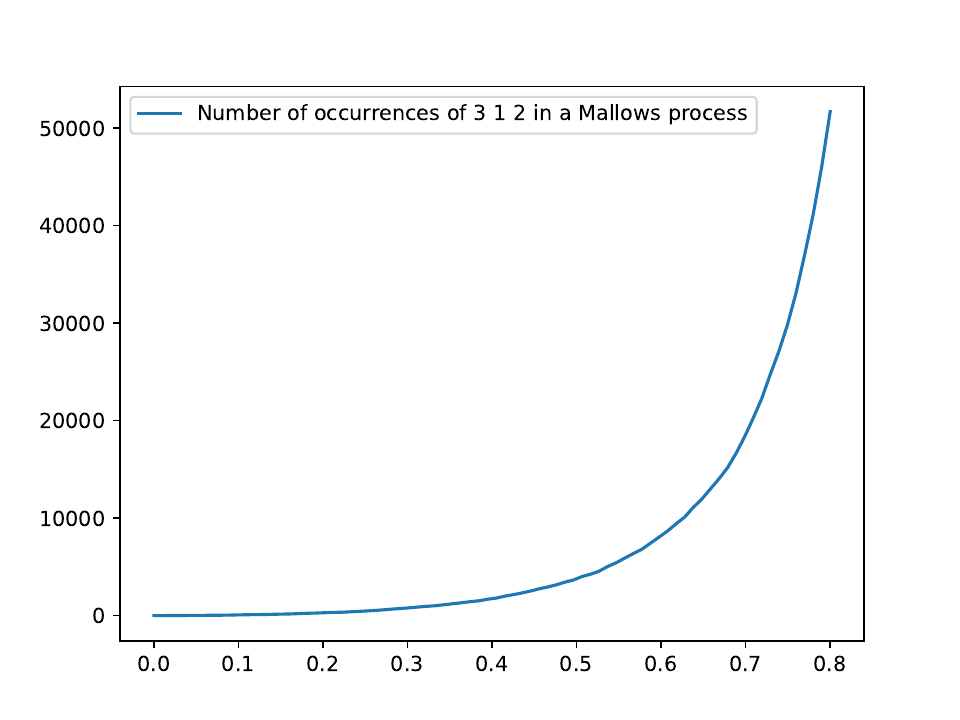}
    \includegraphics[width=.47\linewidth]{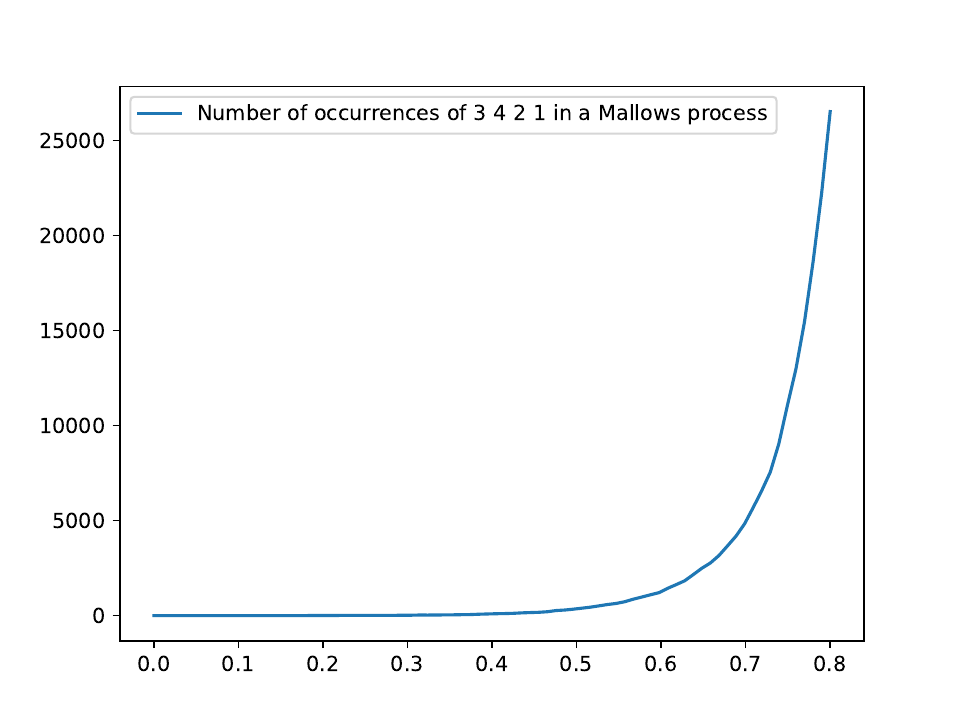}
    \caption{The first two figures show simulations of $\left( \occ{\pi}{\tau_{n,t}} \right)_{t\in(0,q]}$ and $\left( \occ{\pi}{\tau_{n,t}} - \binom{n}{2} e_{\pi,t} \right)_{t\in(0,q]}$ for $\pi=213$.
    The other two show simulations of $\left( \occ{\pi}{\tau_{n,t}} \right)_{t\in(0,q]}$ for $\pi\in \{312, 3421\}$, for which the expression of $e_{\pi,t}$ is unknown.
    This is done with $n=5000$ and up to time $q=0.8$.}
    \label{fig: occ size3 process}
\end{figure}

Finally, we study the regularity of the limit processes.

\begin{proposition}\label{prop: continuity}
    For any pattern $\pi\in\S$, the following holds:
    \begin{itemize}
        \item The function $t\in[0,1) \mapsto e_{\pi,t}$ is locally $\gamma$-H\"older-continuous, for every $\gamma\in(0,1)$.
        \item The Gaussian process $t\in(0,1) \mapsto X_{\pi,t}$ admits a version that is locally $\gamma$-H\"older-continuous, for every $\gamma\in(0,\tfrac12)$.
    \end{itemize}
\end{proposition}

The proof of \Cref{th: Donsker for time-joint single pattern} is conducted in \Cref{sec: proof time-joint CLT}.
It uses once again the tools of \cite{J18}, as for the proof of \Cref{th: LLN and CLT for joint patterns with fixed q}.
Then, in \Cref{sec: proof continuity}, we establish some estimates on the number of \enquote{jumps} occuring in $\left( \Sigma_t \right)_{t\in[0,1)}$, before proving \Cref{prop: continuity}.

\subsection{Open questions}

There is still much work to be done in the study of classical pattern occurrences in Mallows random permutations.
Below, we list a few questions for further research on this topic.

\begin{question}
    In the \enquote{thermodynamic} regime $n(1-q) \to \beta\in(0,\infty)$, the Mallows distribution has a non-trivial permuton limit.
    \Cref{eq: first order in thermodynamic regime} shows a phase transition for the first-order asymptotics of $\occ{\pi}{\tau_n}$, but what are the asymptotic fluctuations?
\end{question}

\begin{question}
    In \Cref{th: order of occ for transition regime} we gave a rough estimate for the first-order asymptotics of $\occ{\pi}{\tau_n}$, but what are the next terms in its expansion?
    Does a central limit theorem hold under appropriate assumptions?
\end{question}

\begin{question}
    Can we find an explicit formula for $e_{\pi,q}$, for general patterns $\pi$?
\end{question}

\begin{question}
    \Cref{th: Donsker for time-joint single pattern} uses a Mallows process $(\tau_{n,t})_{t\in[0,1)}$ whose definition is partly motivated by \Cref{th: continuous-time infinite Mallows process,th: CTI Mallows process as sum of blocks}.
    It would be interesting to relate it to the processes studied in \cite{C22}; it is clearly smooth, but is it regular?
    If not, could we prove an analog of \Cref{th: Donsker for time-joint single pattern} for regular Mallows processes?
\end{question}

\begin{question}
    In \Cref{th: Donsker for time-joint single pattern} and \Cref{prop: continuity}, we proved a finite-dimensional convergence on $[0,1)$ towards a Gaussian process $X_{\pi,\cdot}$ which is continuous on the open interval $(0,1)$.
    Does $X_{\pi,\cdot}$ admit a continuous version on the semi-closed interval $[0,1)$, and does the convergence hold in the Skorohod space?
    Note that when the limit process is continuous, this corresponds to uniform convergence on all compact subsets of $[0,1)$.
\end{question}

\section{Preliminary results on inversion counts}\label{sec: preliminaries}

In this section we collect a few classical facts about the left- and right-inversion counts of permutations, as these are essential for the study of Mallows distributions.

Let $\sigma\in\S_n$.
We define its left- and right-inversion counts as 
\[
\ell_i(\sigma) = \card{ j\in[n] \,:\, j<i ,\, \sigma(j) > \sigma(i) }
\quad;\quad
r_i(\sigma) = \card{ j\in[n] \,:\, j>i ,\, \sigma(j) < \sigma(i) }
\]
for $i\in[n]$.
Clearly, $\sum_{i=1}^n \ell_i(\sigma) = \sum_{i=1}^n r_i(\sigma) = \inv{\sigma}$.
The sequence $\left(\ell_i(\sigma)\right)_{1\le i\le n}$ is usually called the {\em inversion sequence} of $\sigma$; we avoid using this name here, since the sequence $\left(r_i(\sigma)\right)_{1\le i\le n}$ is just as important for our study.

The maps
\[
\varphi_n : \sigma\in\S_n \mapsto \left( \ell_i(\sigma) \right)_{1\le i\le n}
\quad;\quad
\psi_n : \sigma\in\S_n \mapsto \left( r_i(\sigma) \right)_{1\le i\le n}
\]
are both one-to-one, and their ranges are given by:

\begin{lemma}\label{lem: ranges of phi_n and psi_n}
    Let $(\ell_i)_{1\le i\le n}$ and $(r_i)_{1\le i\le n}$ be sequences of nonnegative integers. Then:
    \[
    (\ell_i)_{1\le i\le n} \in \varphi_n\left( \S_n \right) 
    \,\Leftrightarrow\,
    \forall i,\, \ell_i \in \{0,1,\dots,i\m 1\}
    \quad;\quad
    (r_i)_{1\le i\le n} \in \psi_n\left( \S_n \right) 
    \,\Leftrightarrow\,
    \forall i,\, r_i \in \{0,1,\dots,n\m i\} .
    \]
\end{lemma}

\noindent The respective preimages of $\varphi_n$ and $\psi_n$ can be constructed from left to right by two insertion procedures:
\begin{itemize}
    \item If $(\ell_i)_{1\le i\le n} = \varphi_n(\sigma)$, we can recover $\sigma$ by successively constructing $\sigma^{(k)} = \pat{[k]}{\sigma}$ for $k=1,2,\dots,n$.
    First, set $\sigma^{(1)} = 1$.
    Then, if $\sigma^{(k)}\in\S_k$ is constructed for some $k\in[n\m1]$, define $\sigma^{(k+1)}\in\S_{k+1}$ by setting $\sigma^{(k+1)}(k+1) = k + 1 - \ell_{k+1}$ and $\sigma^{(k+1)}(i) = \sigma^{(k)}(i) + \One{\sigma^{(k)}(i) \ge k+1-\ell_{k+1}}$ for each $i\in[k]$.
    It can be checked that for each $k$, $\sigma^{(k)}$ is the unique permutation with left-inversion counts $(\ell_i)_{1\le i\le k}$, and in particular $\sigma^{(n)} = \sigma$.
    \item If $(r_i)_{1\le i\le n} = \psi_n(\sigma)$, we can recover $\sigma$ by successively constructing $\sigma(k)$ for $k=1,2,\dots,n$.
    First, set $\sigma(1) = r_1\p1$.
    Then, if $\sigma(1),\dots,\sigma(k)$ are constructed for some $k\in[n-1]$, define $\sigma(k+1)$ as the $(r_{k+1} \p 1)$st lowest number in the set $[n] \setminus \left\{ \sigma(1),\dots,\sigma(k) \right\}$.
\end{itemize}
Note that we could opt for right-to-left constructions by exchanging the two algorithms, since the sequence $\varphi_n\left( \sigma^\rev \right)$ is obtained by reversing the sequence $\psi_n(\sigma)$.

These notions can be extended to infinite permutations, i.e.~bijections of $\N^*$.
We define the left- and right-inversion counts of $\Sigma\in\S_\infty$ as 
\[
\ell_i(\Sigma) = \card{ j\in\N^* \,:\, j<i ,\, \Sigma(j) > \Sigma(i) }
\quad;\quad
r_i(\Sigma) = \card{ j\in\N^* \,:\, j>i ,\, \Sigma(j) < \Sigma(i) }
\]
for $i\in\N^*$.
The maps
\[
\varphi_\infty : \Sigma\in\S_\infty \mapsto \left( \ell_i(\Sigma) \right)_{i\ge1}
\quad;\quad
\psi_\infty : \Sigma\in\S_\infty \mapsto \left( r_i(\Sigma) \right)_{i\ge1}
\]
are again one-to-one, but their ranges are rather different from the finite case.
Let us focus on $\psi_\infty$:

\begin{lemma}\label{lem: range of psi_infty}
    Let $(r_i)_{i\ge1}$ be a sequence of nonnegative integers. Then:
    \[
    (r_i)_{i\ge1} \in \psi_\infty\left( \S_\infty \right) 
    \,\Longleftrightarrow\,
    r_i=0 \text{ for infinitely many $i$'s}.
    \]
\end{lemma}

\noindent The map $\psi_\infty$ can be inverted by an infinite analog of the previous left-to-right insertion procedure, namely:
for each $k\in\N^*$, $\Sigma(k)$ is the $\left( r_k(\Sigma)\p1 \right)$st lowest number in the set $\N^* \setminus \left\{ \Sigma(1),\dots,\Sigma(k\m1) \right\}$.
This procedure is the only one we will use in this paper;
the infinite case (on $\N^*$) introduces an asymmetry between $\varphi_\infty$ and $\psi_\infty$, and the range of the former is less explicit (see \cite[Remark~3.10]{AK24} for a related discussion).

\medskip

Left- and right-inversion counts are fundamental tools in the study of Mallows random permutations.
Indeed, for $k\in\N^*$ and $q\in[0,1)$, define the truncated geometric distribution $\TGeom{k}{1\m q}$ on $[k]$ by giving mass $\frac{(1-q)q^{j-1}}{1-q^k}$ to each integer $j\in[k]$.
Equivalently, the $\TGeom{k}{1\m q}$ distribution is the $\Geom{1\m q}$ distribution conditioned to the set $[k]$.
Then:

\begin{lemma}[\cite{M57}]\label{lem: left- and right-inversion numbers of finite Mallows}
    Let $\tau_n \sim \Mallows{n}{q}$.
    The random variables $L_i := \ell_i(\tau_n)+1$ are independent for $i\in[n]$ and have distribution $\TGeom{i}{1\m q}$.
    Similarly, the random variables $R_i := r_i(\tau_n)+1$ are independent for $i\in[n]$ and have distribution $\TGeom{n-i+1}{1\m q}$.
\end{lemma}

\noindent As is usual in probability theory, describing a random variable as the image of independent r.v.'s through a simple function proves to be very useful.
The infinite Mallows distribution is defined in an analogous way; see \Cref{sec: infinite Mallows q fixed}.

We end this section with a simple but handy description of decomposability via right-inversion counts.

\begin{lemma}\label{lem: CNS right-inversion decomposable finite}
    Let $\sigma\in\S_n$ with right-inversion counts $(r_i)_{1\le i\le n}$, and let $M\in[n]$.
    We can write $\sigma = \rho \oplus \sigma'$ for some $\rho\in\S_M$ and $\sigma'\in\S_{n-M}$ if and only if:
    \begin{equation}\label{eq: CNS right-inversion decomposable finite}
        \text{for all } i\in[M],\quad
        r_i \le M-i .
    \end{equation}
\end{lemma}

\begin{proof}
    First suppose that $\sigma = \rho\oplus \sigma'$.
    By definition of the direct sum, no couple $(i,j)$ with $i\le M$ and $j>M$ can create an inversion for $\sigma$.
    Thus $r_i(\sigma) = r_i(\rho)$ for all $i\in[M]$, and \Cref{eq: CNS right-inversion decomposable finite} follows.

    Reciprocally, suppose that $\eqref{eq: CNS right-inversion decomposable finite}$ holds.
    By \Cref{lem: ranges of phi_n and psi_n}, there exists $\rho\in\S_M$ such that $r_i(\rho) = r_i(\sigma)$ for all $i\in[M]$.
    Recall that the insertion procedure constructs successively $\sigma(1), \dots, \sigma(n)$ using $r_1, \dots, r_n$.
    The first $M$ steps of the procedure are the same for $\sigma$ as they are for $\rho$, thus $\sigma(i) = \rho(i)$ for all $i\in[M]$ as desired.
\end{proof}

In the same way, an analogous lemma for the infinite case readily follows:

\begin{lemma}\label{lem: CNS right-inversion decomposable infinite}
    Let $\Sigma \in\S_\infty$ with right-inversion counts $(r_i)_{i\ge 1}$.  
    We can write $\Sigma = \rho_1\oplus \dots\oplus \rho_k\oplus \dots$ for some infinite sequence of finite permutations $\rho_k \in \S_{M_k}$, $k\ge1$, if and only if $(r_i)_{i\ge 1}$ is an infinite concatenation of sequences satisfying \eqref{eq: CNS right-inversion decomposable finite}, that is:
    \begin{equation*}
        \text{for all } k\in\N^* \text{ and } i\in[M_k], \quad
        r_{M_1+ \dots+ M_{k\m1}+ i} \le M_k - i .
    \end{equation*}
\end{lemma}

\section{The almost uniform regime: coupling and stability}\label{sec: proofs almost uniform}

This section is devoted to the proofs of \Cref{lem: coupling Mallows with uniform} and \Cref{th: stability close to uniform}.
We start with two useful lemmas.

\begin{lemma}\label{lem: punctual dTV Tgeom with Unif}
    Let $0\le q\le 1$ and $m\in\N^*$ such that $m(1-q)\le1$.
    Let $X_m \sim \TGeom{m}{1-q}$ and $Y_m \sim \Unif{[m]}$.
    Then for any $j\in[m]$:
    \begin{equation*}
        \abs{\big. \prob{X_m=j} - \prob{Y_m=j} } = \abs{ \prob{X_m=j} - \frac1m } \le 3(1-q) .
    \end{equation*}
\end{lemma}

\begin{proof}
    This can be found in the proof of \cite[Lemma~7.2]{MSV23}, but we include the computation here for completeness.
    Recall that for any $x\in[0,1]$ and $j\in\N$, we have $\abs{(1-x)^j - 1} \le jx$ and $\abs{(1-x)^j - (1-jx)} \le \tfrac12 j^2x^2$.
    Therefore we can write $q^j = 1 + \Landau{\O_1}{}{j(1-q)}$ and $q^j = 1 -j(1-q) + \Landau{\O_{1/2}}{}{j^2(1-q)^2}$ where the notation $\O_\delta(t)$ means that $\abs{\O_\delta(t)} \le \delta t$.
    Hence, for all $j\in[m]$:
    \begin{align*}
        \abs{ \prob{X_m=j} - \prob{Y_m=j} }
        &= \abs{ \frac{(1-q)q^{j-1}}{1-q^m} - \frac1m } \\
        &= \frac{\abs{ m(1-q)q^{j-1} - (1-q^m) }}{m(1-q^m)} \\
        &\le \frac{\abs{ m(1\m q)\left( 1 + \Landau{\O_1}{}{j(1\m q)} \right) - m(1\m q) + \Landau{\O_{1/2}}{}{m^2(1\m q)^2} }}{m\left( m(1\m q) - \tfrac12 m^2(1\m q)^2 \right)} \\
        &\le \frac{ \Landau{\O_{3/2}}{}{m^2(1-q)^2} }{ m^2(1-q) - \tfrac12 m^3(1-q)^2 } \\
        &\le \frac{ \tfrac32 (1-q) }{ 1 - \tfrac12 m(1-q) } \\
        &\le 3 (1-q) .\qedhere
    \end{align*}
\end{proof}

In the following lemma, we write $\dtv{\cdot}{\cdot}$ for the total variation distance, and $\phi^*\cL$ for the pushforward of a probability distribution $\cL$ by a measurable function $\phi$.

\begin{lemma}\label{lem: coupling dtv pushforward}
    Let $\cL_1,\cL_2$ be probability distributions on $\R$, and $\phi_1,\phi_2 : \R \to \N$ be measurable functions.
    Let $\epsilon>0$ and suppose that $\dtv{\phi_1^*\cL_1}{\phi_2^*\cL_2} \le \epsilon$.
    Then there exists a coupling $(X_1,X_2)$ of $\cL_1$ and $\cL_2$ such that
    \begin{equation*}
        \prob{ \phi_1(X_1) \ne \phi_2(X_2) } \le \epsilon .
    \end{equation*}
\end{lemma}

\begin{proof}
    It is well-known that there exists a coupling $(Y_1,Y_2)$ of $\phi_1^*\cL_1$ and $\phi_2^*\cL_2$ such that 
    \begin{equation*}
        \prob{ Y_1 \ne Y_2 } \le \epsilon .
    \end{equation*}
    Let $i\in\{1,2\}$.
    For each $y$ in the support of $\phi_i^*\cL_i$, define $\cL_{i,y}$ as the probability distribution induced by $\cL_i$ on the fiber $\phi_i^{-1}(\{y\})$.
    Then let $X_i$ have distribution $\cL_{i,Y_i}$, conditionally given $Y_i$.
    It is readily verified that $X_i$ is a.s.~well defined, and that its law is $\cL_i$.
    Furthermore $\phi_i(X_i) = Y_i$ a.s., whence the desired inequality.
\end{proof}

\begin{proof}[Proof of \Cref{lem: coupling Mallows with uniform}]
    In this proof, we opt for the construction of $\tau_n \sim \Mallows{n}{q}$ via its {left-inversions}: let $L_1,\dots,L_n$ be independent r.v.'s such that for each $k\in[n]$, $L_k \sim \TGeom{k}{1\m q_n}$.
    We define, recursively for $1\le k\le n$, $\tau_n^k$ as the only permutation of $[k]$ such that $\tau_n^k(k) = k-L_k+1$ and $\pat{[k\m1]}{\tau_n^k} = \tau_n^{k-1}$.
    Then $\tau_n := \tau_n^n \sim \Mallows{n}{q}$ by \Cref{lem: left- and right-inversion numbers of finite Mallows}, and $\pat{[k]}{\tau_n} = \tau_n^k$ for all $k$.
    Similarly, we can construct $u_n \sim \Unif{\S_n}$ using independent r.v.'s $U_k \sim \Unif{[k]}$ for $k\in[n]$.
    
    Our goal is to couple the $L_k$'s with the $U_k$'s in such a way that some large pattern of $\tau_n$ coincides with some large pattern of $u_n$.
    More precisely, recursively for $1\le k\le n$, we shall construct a well-chosen coupling of $L_k$ and $U_k$ conditionally given the $\sigma$-algebra $\cF_{k-1} := \sigma\left( L_1,U_1,\dots,L_{k-1},U_{k-1} \right)$.
    Then the marginal laws of $L_k$ and $U_k$, conditionally given $\cF_{k-1}$, will be $\TGeom{k}{1\m q_n}$ and $\Unif{[k]}$, thus the associated permutations $\tau_n$ and $u_n$ will indeed have distributions $\Mallows{n}{q_n}$ and $\Unif{\S_n}$.
    In parallel to this, we shall recursively construct an $\cF_k$-measurable set $\cH_k \subset [k]$ of \enquote{errors}, satisfying $\pat{[k]\setminus\cH_k}{\tau_n^k} = \pat{[k]\setminus\cH_k}{u_n^k}$.

    First, set $L_1 = U_1 = 1$ a.s.~and $\cH_1=\emptyset$.
    Then, assume that we have constructed $L_1,U_1,\dots,L_k,U_k$ and $\cH_k$ as desired, for some $k<n$.
    If $\sigma\in\S_k$, define $\phi_\sigma : x\in[k+1] \mapsto x + \sum_{h\in\cH_k} \One{x < k-\sigma(h)+2}$ (this definition is motivated later in the proof).
    The function $\phi_\sigma$ is nondecreasing and onto $[1+H_k,k+1]$.
    In particular, $\sum_{j=1+H_k}^{k+1} \left( \abs{\phi_\sigma^{-1}(\{j\})}\m1 \right) = H_k$.
    Below, for notational convenience, we allow ourselves to write $L_{k+1}$ and $U_{k+1}$ for generic variables with laws $\TGeom{k\p1}{1\m q_n}$ and $\Unif{[k\p1]}$, although their coupling has not yet been defined.
    Henceforth, conditionally given $\cF_k$ and using \Cref{lem: punctual dTV Tgeom with Unif} on the fourth line:
    \begin{align*}
        \dtv{\phi_{\tau_n^k}\left( L_{k+1} \right)}{\phi_{u_n^k}\left( U_{k+1} \right)} 
        &= \tfrac12 \sum_{j=1+H_k}^{k+1} \abs{ \prob{\phi_{\tau_n^k}\left( L_{k+1} \right)=j} - \prob{\phi_{u_n^k}\left( U_{k+1} \right)=j} } \\ 
        &= \tfrac12 \sum_{j=1+H_k}^{k+1} \abs{ \prob{ L_{k+1} \in \phi_{\tau_n^k}^{-1}(\{j\})} - \frac{ \abs{\phi_{u_n^k}^{-1}(\{j\})} }{k+1} } \\
        &\le \tfrac12 \sum_{j=1+H_k}^{k+1} \left(
        \abs{ \prob{ L_{k+1} \in \phi_{\tau_n^k}^{-1}(\{j\})} - \frac{ \abs{\phi_{\tau_n^k}^{-1}(\{j\})} }{k+1} } 
        + \frac{\abs{ \abs{\phi_{u_n^k}^{-1}(\{j\})} \m \abs{\phi_{\tau_n^k}^{-1}(\{j\})} }}{k+1}
        \right)\\
        &\le \tfrac12 \sum_{j=1+H_k}^{k+1} \left(
        3(1-q) \abs{\phi_{\tau_n^k}^{-1}(\{j\})} 
        + \frac{ \abs{\phi_{u_n^k}^{-1}(\{j\})}\m1 + \abs{\phi_{\tau_n^k}^{-1}(\{j\})}\m1 }{k+1}
        \right)\\
        &\le \tfrac12 \left( 3(1-q)\left( k+1 \right) + \frac{2H_k}{k+1} \right)\\
        &\le \tfrac32 (k+1) (1-q) + \frac{H_k}{k+1} .
    \end{align*}
    Thus we can apply \Cref{lem: coupling dtv pushforward}:
    there exists a coupling of $L_{k+1}$ and $U_{k+1}$, conditionally given $\cF_k$, such that 
    \begin{equation}\label{eq: coupling proba pushforward phi}
        \prob{ \phi_{\tau_n^k}(L_{k+1}) \ne \phi_{u_n^k}(U_{k+1}) } \le \tfrac{3}{2}(k+1)(1-q) + \frac{H_k}{k+1} .
    \end{equation}
    Now, note that the last values of $\pat{[k\p1]\!\setminus\!\cH_k}{\tau_n^{k+1}}$ and $\pat{[k\p1]\setminus\cH_k}{u_n^{k+1}}$ are given by
    \begin{align*}
        \left\{
        \begin{array}{ll}
        \pat{[k\p1]\setminus\cH_k}{\tau_n^{k+1}}(-1) = \tau_n^{k+1}(k\p1) - \sum_{h\in\cH_k} \One{\tau_n^{k}(h) < \tau_n^{k+1}(k+1)} ;\smallskip\\
        \pat{[k\p1]\setminus\cH_k}{u_n^{k+1}}(-1) = u_n^{k+1}(k\p1) - \sum_{h\in\cH_k} \One{u_n^{k}(h) < u_n^{k+1}(k+1)} ;
        \end{array}
        \right.
    \end{align*}
    using the notation $\sigma(-1) = \sigma(\ell-1)$ for $\sigma\in\S_\ell$.
    Recalling that $\tau_n^{k+1}(k+1) = k+2-L_{k+1}$ and likewise $u_n^{k+1}(k+1) = k+2-U_{k+1}$, the previous two numbers are equal if and only if
    \begin{equation*}
        L_{k+1} + \sum_{h\in\cH_k} \One{L_{k+1} < k-\tau_n^{k}(h)+2}
        = U_{k+1} + \sum_{h\in\cH_k} \One{U_{k+1} < k-u_n^{k}(h)+2} \,,
    \end{equation*}
    i.e.~if $\phi_{\tau_n^k}\left( L_{k+1} \right) = \phi_{u_n^k}\left( U_{k+1} \right)$.
    When this is the case, since we assumed that $\pat{[k]\!\setminus\!\cH_k}{\tau_n^k} = \pat{[k]\!\setminus\!\cH_k}{u_n^k}$, we deduce that $\pat{[k+1]\setminus\cH_k}{\tau_n^{k+1}} = \pat{[k+1]\setminus\cH_k}{u_n^{k+1}}$ and we can set $\cH_{k+1} = \cH_k$.
    Otherwise, if $\phi_{\tau_n^k}\left( L_{k+1} \right) \ne \phi_{u_n^k}\left( U_{k+1} \right)$, we set $\cH_{k+1} = \cH_k \cup \{k+1\}$ and we indeed have 
    \[
    \pat{[k\p1]\setminus\cH_{k+1}}{\tau_n^{k+1}} 
    = \pat{[k]\setminus\cH_{k}}{\tau_n^{k}}
    = \pat{[k]\!\setminus\cH_{k}}{u_n^{k}}
    = \pat{[k\p1]\setminus\cH_{k+1}}{u_n^{k+1}}
    .\]
    This coupling is illustrated in \Cref{fig: coupling example}.

    \begin{figure}
        \centering
        \includegraphics[scale=.8]{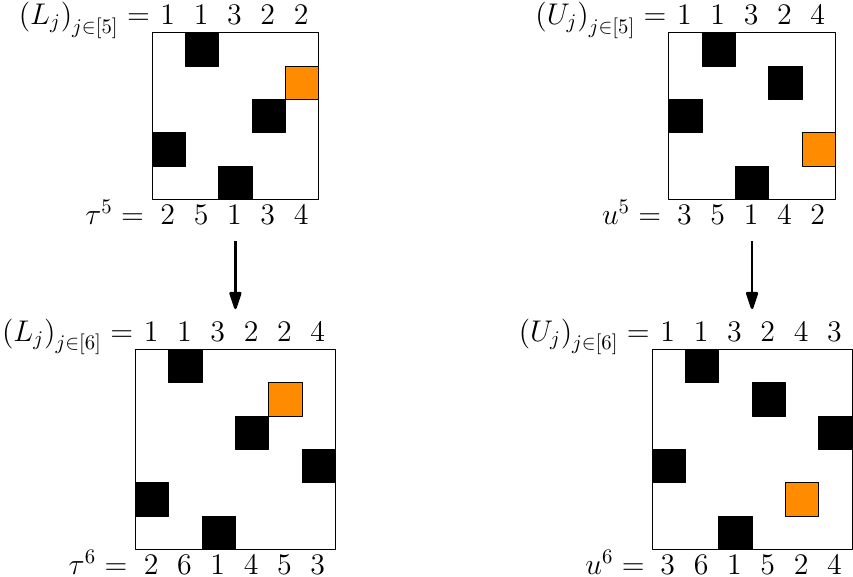}
        \caption{Illustration of our coupling, and how it adapts to errors.
        Here, the first error happens at step $k=5$ (in orange). 
        The following left-inversion counts are coupled using the functions $\phi_{\tau^5} : x \mapsto x + \One{x<3}$ and $\phi_{u^5} : x \mapsto x + \One{x<5}$.
        No error happens at step $6$, since $\phi_{\tau^5}(L_6) = 4 = \phi_{u^5}(U_6)$.
        Observe that $\pat{[6]\setminus\{5\}}{\tau^6} = \pat{[6]\setminus\{5\}}{u^6}$, as expected.
        }
        \label{fig: coupling example}
    \end{figure}

    We have thus defined our coupling of $\tau_n$ and $u_n$, and need to study the total number of errors $H_n$.
    Using \eqref{eq: coupling proba pushforward phi}, we have for each $k<n$:
    \begin{align*}
        \expecond{H_{k+1}}{\cF_k} \le H_k + \tfrac{3}{2} (k+1) (1-q) + \frac{H_k}{k+1}
        = H_k \frac{k+2}{k+1} + \tfrac{3}{2} (k+1) (1-q) ,
    \end{align*}
    whence
    \begin{align*}
        \expecond{\frac{H_{k+1}}{k+2}}{\cF_k} \le \frac{H_k}{k+1} + \tfrac{3}{2} (1-q) .
    \end{align*}
    We deduce by induction that
    \begin{align}\label{eq: bound mean H_k}
        \expec{\frac{H_{k+1}}{k+2}} \le \tfrac{3}{2} (1-q) k 
    \end{align}
    since $H_1=0$.
    Applied with $k+1 = n$, this yields $\expec{H_{n}} \le \tfrac{3}{2} n^2 (1-q)$.
\end{proof}

\begin{proof}[Proof of \Cref{th: stability close to uniform}]
    Using \Cref{lem: coupling Mallows with uniform}, this is straightforward.
    
    Since $n^{3/2}(1-q)\to0$, we have $H_n / \sqrt{n} \cv{} 0$ in $L^1$ as $n\to\infty$, hence $H_n = \Landau{o_{\P}}{}{\sqrt{n}}$.
    Furthermore, $\check \tau_n = \check u_n$ where $\check \tau_n := \pat{[n]\setminus\cH_n}{\tau_n}$ and $\check u_n := \pat{[n]\setminus\cH_n}{u_n}$.
    By bounding the number of subsets of $[n]$ having a non-empty intersection with $\cH_n$, we have:
    \begin{equation*}
        \abs{ \occ{\pi}{\tau_n} - \occ{\pi}{\check \tau_n} } \le n^{r-1} H_n
    \end{equation*}
    and likewise for $u_n$, thus:
    \begin{equation*}
        \occ{\pi}{\tau_n} = \occ{\pi}{u_n} + \Landau{o_{\P}}{}{n^{r-1/2}} ,
    \end{equation*}
    where $\Landau{o_{\P}}{}{n^{r-1/2}}$ is uniform in $\pi$.
    Using \Cref{th: CLT occ unif}, this proves \Cref{th: stability close to uniform}.
\end{proof}

\section{The transition regime: near decomposability}\label{sec: proofs transition}

This section is devoted to the proof of \Cref{th: order of occ for transition regime}.
Our proof requires good control over the displacements of points in random Mallows permutations, specifically:

\begin{theorem}[\cite{BP15}]\label{th: concentration inequality displacements}
    Let $q\in(0,1)$ and $\tau \sim \Mallows{n}{q}$.
    Then for all $i\in[n]$ and $t>0$:
    \begin{equation*}
        \prob{ \abs{\tau(i)-i} \ge t } \le 2 q^t .
    \end{equation*}
\end{theorem}

\begin{figure}
    \centering
    \includegraphics[scale=.86]{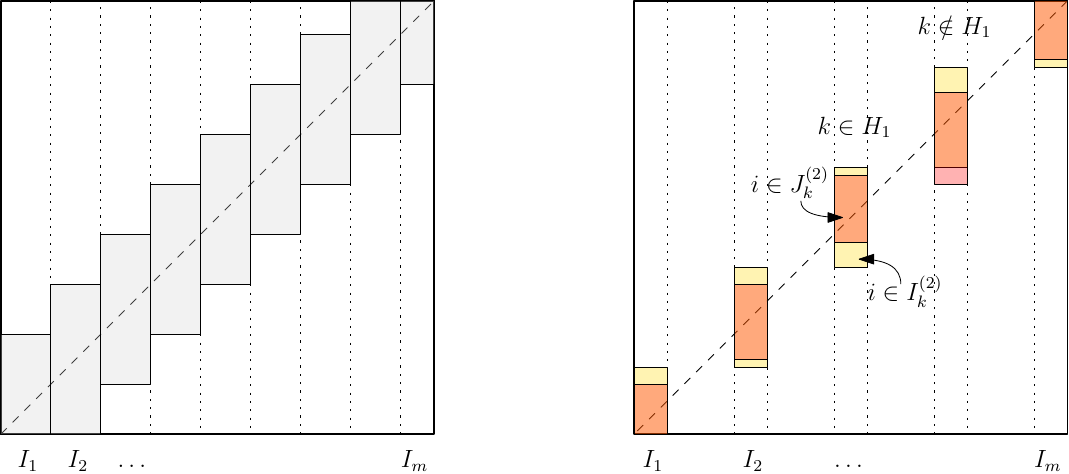}
    \caption{
    On the left: the intervals defined in the proof of the upper bound of \Cref{th: order of occ for transition regime}. 
    Under the event $E$, all permutation points are contained in the grey area.
    On the right: the intervals defined in the proof of the lower bound of \Cref{th: order of occ for transition regime}.
    The points with index in $I_k^{(2)}$, resp.~$J_k^{(2)}$, are in the yellow area, resp.~red area.
    With high probability, a positive proportion of points in the red area are in the yellow area.
    }
    \label{fig: transition intervals}
\end{figure}

\begin{proof}[Proof of \Cref{th: order of occ for transition regime}]
    We prove separately an upper and a lower bound on $\occ{\pi}{\tau_n}$, with similar techniques.
    The intervals used in the proof are illustrated in \Cref{fig: transition intervals}.
    
    \paragraph{Upper bound:}
    Consider an arbitrary sequence $(\lambda_n)$ such that $\lambda_n - \log n \to +\infty$ as $n\to\infty$, and define the event $E := \left\{ \forall i\in[n],\quad \abs{\tau_n(i)-i} < \frac{\lambda_n}{1-q} \right\}$.
    Then by \Cref{th: concentration inequality displacements} along with a union bound:
    \begin{equation}\label{eq: displacement for upper bound}
        1-\prob{E}
        = \prob{ \exists i\in[n],\quad \abs{\tau_n(i)-i} \ge \frac{\lambda_n}{1-q} }
        \le 2 n\, e^{ \frac{\lambda_n}{1-q} \log q }
        \le 2 n\, e^{-\lambda_n}
        \cv{n\to \infty} 0.
    \end{equation}
    This controls the distance of all points $(i,\tau_n(i))$ to the diagonal, and will allow us to restrict the patterns of $\tau_n$ in a sequence of overlapping intervals.
    First set $m := \lceil \frac{n(1-q)}{\lambda_n} \rceil$, and $I_k := \left[ 1 + \lfloor \frac{(k-1) \lambda_n}{1-q} \rfloor \,,\, \lfloor \frac{k \lambda_n}{1-q} \rfloor \wedge n \right]$ for $1\le k\le m$. 
    These disjoint intervals cover $[n]$.
    
    Now, let $\pi_j \in \S_{r_j}$ be a block of $\pi$ and let $1\le i_1 < \dots < i_{r_j}\le n$ be such that $\pat{\{i_1,\dots,i_{r_j}\}}{\tau_n} = \pi_j$.
    For each $\ell \in [r_j]$, define $k_\ell\in[m]$ by $i_\ell \in I_{k_\ell}$.
    Then we claim that under $E$, $\abs{k_{\ell+1} - k_\ell} \le 2$ for all $\ell\in[r_j\m1]$.
    Indeed, suppose that $k_{\ell_0+1} \ge k_{\ell_0}+3$ for some $\ell_0\in[r_j\m1]$.
    Then for any $\ell\le {\ell_0}$ and ${\ell'} \ge {\ell_0+1}$:
    \begin{equation*}
        \tau_n(i_\ell) 
        < i_\ell + \frac{\lambda_n}{1-q}
        \le i_{\ell_0} + \frac{\lambda_n}{1-q}
        \le \frac{(k_{\ell_0}+1)\lambda_n}{1-q}
        \le \frac{(k_{\ell_0+1}-2)\lambda_n}{1-q}
        < i_{\ell_0+1} - \frac{\lambda_n}{1-q} 
        \le i_{\ell'} - \frac{\lambda_n}{1-q} 
        < \tau_n(i_{\ell'}) .
    \end{equation*}
    This contradicts the indecomposability of $\pat{\{i_1,\dots,i_{r_j}\}}{\tau_n} = \pi_j$, and proves the claim.

    This leads us to define $J_k := I_k \cup I_{k+1} \cup \dots \cup I_{k+2(r-1)}$ for each $k\in[m-2(r\m1)]$.
    By the previous claim, if $\pat{\{i_1,\dots,i_{r_j}\}}{\tau_n} = \pi_j$, then there exists $k\in[m-2(r\m1)]$ such that $i_1,\dots,i_{r_j} \in J_k$.
    Subsequently, if $\pat{\{i_1,\dots,i_{r}\}}{\tau_n} = \pi$, then there exist $k_1\le\dots\le k_d$ such that $i_{r_1+\dots+r_{j-1}+1},\dots,i_{r_1+\dots+r_j} \in J_{k_j}$ for all $j\in[d]$.
    This implies the following inequality under $E$:
    \begin{align}\label{eq: deterministic upper bound for occ in transition regime}
        \occ{\pi}{\tau_n}
        &\le \sum_{1\le k_1\le \dots \le k_d \le m-2(r-1)} \, \prod_{1\le j\le d} \occ{\pi_j}{ \pat{J_{k_j}}{\tau_n} } \nonumber\\
        &\le \sum_{1\le k_1\le \dots \le k_d \le m-2(r-1)} \, \prod_{1\le j\le d} \abs{J_{k_j}}^{\abs{\pi_j}} \nonumber\\
        &\le m^d \left( \frac{2 r \lambda_n}{1-q}\right)^{r} \nonumber\\
        &\le (2r)^r (\lambda_n)^{r-d} \frac{n^d}{(1-q)^{r-d}} .
    \end{align}
    To deduce the probability upper bound, set $\lambda_n := 2\log n$, then $\beta := (4r)^r$ and $\gamma := r-d$:
    \begin{equation*}
        \prob{ \occ{\pi}{\tau_n} \le \beta\, (\log n)^\gamma\, \frac{n^d}{(1-q)^{r-d}} }
        \ge \prob{E} \cv{n\to\infty} 1
    \end{equation*}
    by \eqref{eq: displacement for upper bound}.
    To obtain a moment upper bound, we can set $\lambda_n := (rp\p1)\log n$ for any desired $p>0$.
    By \eqref{eq: displacement for upper bound} and \eqref{eq: deterministic upper bound for occ in transition regime}:
    \begin{multline*}
        \expec{ \occ{\pi}{\tau_n}^p }
        \le (2r)^{rp} (\lambda_n)^{rp} \left( \frac{n^d}{(1-q)^{r-d}} \right)^p 
        + n^{rp} \left(\big. 1-\prob{E} \right) \\
        \le (2(rp\p1)r)^{rp} (\log n)^{rp} \left( \frac{n^d}{(1-q)^{r-d}} \right)^p 
        + 2 n^{rp+1} e^{-(rp+1)\log n} .
    \end{multline*}
    The desired upper bound readily follows.
    
    \paragraph{Lower bound:}
    This time, the idea is to consider a disjoint sequence of boxes that contain a good proportion of permutation points.
    The box sizes shall be tuned so that the induced sub-permutations are random Mallows permutations in the thermodynamic regime, and we can apply \eqref{eq: first order in thermodynamic regime}.
    
    Fix $\lambda>0$, then set $m := \lfloor \frac{n(1-q)}{3\lambda} \rfloor$ and $I_k := \left[ 1 + \lfloor \frac{3(k-1)\lambda}{1-q} \rfloor \,,\, \lfloor \frac{3(k-1)\lambda + \lambda}{1-q} \rfloor \right]$ for $1\le k\le m$.
    These are disjoint intervals in $[n]$, and the induced permutation $\tau_{n,k} := \pat{I_k}{\tau_n}$ follows the $\Mallows{n_k}{q}$ distribution for each $k\in[m]$ by \cite[Corollay~2.7]{BP15}, where $n_k := \abs{I_k} = \frac{\lambda}{1-q} + \O(1)$ uniformly in $k$.
    
    We could apply \eqref{eq: first order in thermodynamic regime} to these induced permutations; however, some occurrences of $\pi_j$'s in $\tau_{n,k}$'s might overlap vertically, thus not yielding an occurrence of $\pi$ in $\tau_n$.
    We solve this by cropping the $\tau_{n,k}$'s vertically, as illustrated in \Cref{fig: transition intervals}. 
    For each $k\in[m]$, set
    \begin{align*}
        \left\{
        \begin{array}{lll}
            I_k^{(1)} := \left\{ i\in I_k :\, \tau_n(i) \le \frac{3(k\m1)\lambda - \lambda}{1-q} \right\} ;\smallskip\\
            I_k^{(2)} := \left\{ i\in I_k :\, \frac{3(k\m1)\lambda - \lambda}{1-q} < \tau_n(i) \le \frac{3(k\m1)\lambda + 2\lambda}{1-q} \right\} ;\smallskip\\
            I_k^{(3)} := \left\{ i\in I_k :\, \frac{3(k\m1)\lambda + 2\lambda}{1-q} < \tau_n(i) \right\} ;
        \end{array}
        \right.
        \;
        \left\{
        \begin{array}{lll}
            J_k^{(1)} := \left\{\Big. i\in I_k :\, \tau_{n,k}\left(\,\check i\,\right) \le \lceil \frac13 n_k \rceil \right\} ;\smallskip\\
            J_k^{(2)} := \left\{\Big. i\in I_k :\, \lceil \frac13 n_k \rceil < \tau_{n,k}\left(\,\check i\,\right) \le \lfloor \frac23 n_k \rfloor \right\} ;\smallskip\\
            J_k^{(3)} := \left\{\Big. i\in I_k :\, \lfloor \frac23 n_k \rfloor < \tau_{n,k}\left(\,\check i\,\right) \right\} ;
        \end{array}
        \right.
    \end{align*}
    with the notation $\check i := i - \lfloor \frac{3(k-1)\lambda}{1-q} \rfloor$.
    The reason we introduced $I_k^{(2)}$ is that if $i\in I_k^{(2)}$ and $i'\in I_{k'}^{(2)}$ for $k<k'$, then $\tau_n(i) < \tau_n(i')$. Therefore:
    \begin{equation}\label{eq: lower bound occ transition with Ik2}
        \occ{\pi}{\tau_n} \ge \sum_{1\le k_1<\dots<k_d\le m} \prod_{j=1}^d \occ{\pi_j}{\pat{I_{k_j}^{(2)}}{\tau_n}} .
    \end{equation}
    On the other hand, the definition of $J_k^{(2)}$ is motivated by the fact that the random permutation $\tau_{n,k}^{(2)} := \pat{J_k^{(2)}}{\tau_{n}}$ is $\Mallows{\frac13 n_k \p \O(1)}{q}$ distributed.
    Indeed, it can be written as $\pat{ \left[ 1\p\lceil\frac13 n_k\rceil , \lfloor\frac23 n_k\rfloor \right] }{\tau_{n,k}^{-1}}^{-1}$, which follows the claimed distribution by \cite[Corollaries~2.3 and~2.7]{BP15}.
    We shall now prove that for many $k$'s, $J_k^{(2)} \subseteq I_k^{(2)}$ and $\occ{\pi_j}{\tau_{n,k}}$ has order $n_k^{\abs{\pi_j}}$.
    
    For each $k\in[m]$ and $i\in I_k$:
    \begin{equation*}
        \prob{i\notin I_k^{(2)}}
        \le \prob{ \abs{\tau_n(i)-i} \ge \frac{\lambda}{1-q} }
        \le 2 e^{-\lambda}
    \end{equation*}
    by \Cref{th: concentration inequality displacements}.
    Hence $\expec{\abs{I_k^{(1)}\cup I_k^{(3)}}} \le 2 n_k e^{-\lambda}$, and by Markov's inequality:
    \begin{equation*}
        \prob{\abs{I_k^{(1)}\cup I_k^{(3)}} \ge \frac13 n_k}
        \le 6 e^{-\lambda} .
    \end{equation*}
    Thus $H_1 := \left\{ k\in[m] : \abs{I_k^{(1)}\cup I_k^{(3)}} \ge \frac13 n_k \right\}$ satisfies $\expec{\abs{H_1}} \le 6 m e^{-\lambda}$, and:
    \begin{equation}\label{eq: control size H_1}
        \prob{ \abs{H_1} \ge \frac14 m } \le 24 e^{-\lambda} .
    \end{equation}
    Let $k\in[m]$ and suppose that $k\notin H_1$.
    Then $\abs{I_k^{(1)}} < \frac13 n_k$.
    Since $\abs{J_k^{(1)}} \ge \frac13 n_k$, and by construction, we deduce that $I_k^{(1)} \subset J_k^{(1)}$.
    Likewise, it holds that $I_k^{(3)} \subset J_k^{(3)}$.
    Subsequently $J_k^{(2)} \subset I_k^{(2)}$, and \eqref{eq: lower bound occ transition with Ik2} implies:
    \begin{equation}\label{eq: lower bound occ transition with Jk2}
        \occ{\pi}{\tau_n}
        \ge \sum_{1\le k_1 < \dots < k_d\le m} \prod_{j=1}^d \occ{\pi_j}{\tau_{n,k_j}^{(2)}} \One{k_j \notin H_1} .
    \end{equation}
    
    Write $r_j := \abs{\pi_j}$ for $j\in[d]$.
    Recall that $\tau_{n,k}^{(2)} \sim \Mallows{\frac13 n_k + \O(1)}{q}$, where $\frac13 n_k (1-q) \to \frac13 \lambda$ as $n\to\infty$.
    Then by \eqref{eq: first order in thermodynamic regime}, uniformly in $j\in[d]$ and $k\in[m]$:
    \begin{equation*}
        \prob{ \occ{\pi_j}{\tau_{n,k}^{(2)}} \le \frac12 \dens{\pi_j}{\mu_{\lambda/3}} \frac{n_k^{r_j}}{3^{r_j} r_j !} } \cv{n\to\infty} 0 .
    \end{equation*}
    Set $c_{j,\lambda} := \frac12\frac{\dens{\pi_j}{\mu_{\lambda/3}}}{3^{r_j} r_j !}$ and $H_2 := \left\{ k\in[m] :\, \exists j\in[d],\, \occ{\pi_j}{\tau_{n,k}^{(2)}} \le c_{j,\lambda} n_k^{r_j} \right\}$.
    Then by Markov's inequality: 
    \begin{equation}\label{eq: control size H_2}
        \prob{ \abs{H_2} \ge \frac14 m } 
        \le \frac4m \sum_{1\le k\le m} \sum_{1\le j\le d} \prob{ \occ{\pi_j}{\tau_{n,k}^{(2)}} \le c_{j,\lambda} n_k^{r_j} }
        \cv{n\to\infty} 0 .
    \end{equation}
    Now \eqref{eq: lower bound occ transition with Jk2} yields:
    \begin{align*}
        \occ{\pi}{\tau_n}
        &\ge \sum_{1\le k_1 < \dots < k_d\le m} \prod_{j=1}^d \occ{\pi_j}{\tau_{n,k_j}^{(2)}} \One{k_j \notin H_1} \One{k_j \notin H_2} \\
        &\ge \sum_{1\le k_1 < \dots < k_d\le m} \prod_{j=1}^d c_{j,\lambda} n_{k_j}^{r_j} \One{k_j \notin H_1} \One{k_j \notin H_2} \\
        &\ge \sum_{1\le k_1 < \dots < k_d\le m} \prod_{j=1}^d c_{j,\lambda} \left(\frac{\lambda}{1-q} -2\right)^{r_j} \One{k_j \notin H_1} \One{k_j \notin H_2} \\
        &\ge c_\lambda (1-q)^{-r} \left(m - \abs{H_1} - \abs{H_2}\right)^d 
    \end{align*}
    for some $c_\lambda >0$.
    Therefore, using \eqref{eq: control size H_1} and \eqref{eq: control size H_2} and setting $c_\lambda' := \frac12 c_\lambda / (6\lambda)^d$:
    \begin{equation*}
        \prob{ \occ{\pi}{\tau_n} \ge c_\lambda' \frac{n^d}{(1-q)^{r-d}} } 
        \ge \prob{ \abs{H_1}+\abs{H_2} < \frac12 m }
        \ge 1 - 24e^{-\lambda} - o(1)
    \end{equation*}
    Since $24e^{-\lambda}$ goes to $0$ as $\lambda\to\infty$, this concludes the proof of the probability lower bound.
    The moment lower bound is then straightforward:
    \begin{equation*}
        \expec{ \occ{\pi}{\tau_n}^p }
        \ge \left( 1 - 24e^{-\lambda} - o(1) \right) \left( c_\lambda' \frac{n^d}{(1-q)^{r-d}} \right)^p ,
    \end{equation*}
    for any fixed $\lambda>0$ such that $24e^{-\lambda} < 1$.
\end{proof}

\begin{remark}
    In the lower bound of the previous proof, we could replace the sub-permutations $\pat{J_k^{(2)}}{\tau_n}$ by $\pat{I_k^{(2)}}{\tau_n}$, which follow the Mallows distribution by \cite[Lemma~3.15]{GP18}.
    In this case, instead of proving that many $J_k^{(2)}$'s are in $I_k^{(2)}$'s, it should be shown (by the same arguments) that many $I_k^{(2)}$'s contain $\Omega(n_k)$ permutation points.
\end{remark}

\section{The regenerative regime}

\subsection{Infinite Mallows permutations}\label{sec: infinite Mallows q fixed}

The authors of \cite{GO10} defined an infinite version of the Mallows distribution, as a random element of $\S_\infty$.
In this section we recall their construction, and present a few key properties.

Fix $q\in[0,1)$, and let $\left( R_i \right)_{i\ge1}$ be a sequence of i.i.d.~$\Geom{1\m q}$ r.v.'s.
The sequence $(R_i-1)_{i\ge1}$ almost surely satisfies the condition of \Cref{lem: range of psi_infty}, and we say that the random permutation $\Sigma \in \S_\infty$ with right-inversion counts $(R_i-1)_{i\ge1}$ follows the $\Mallows{\N^*}{q}$ distribution.
An important property of the $\Mallows{\N^*}{q}$ measure is that it can be seen as the projective limit of $\Mallows{n}{q}$ measures, in the following sense \cite[Proposition~A.6]{GO10}:

\begin{lemma}\label{lem: pattern of infinite Mallows}
    If $\Sigma \sim \Mallows{\N^*}{q}$ then for any $n\ge1$, the finite sub-permutation $\tau_n := \pat{[n]}{\Sigma}$ follows the $\Mallows{n}{q}$ distribution.
\end{lemma}

\Cref{lem: pattern of infinite Mallows} allows us to define $\tau_n \sim \Mallows{n}{q}$ as a sub-permutation of $\Sigma \sim \Mallows{\N^*}{q}$, and to study its asymptotic properties as $n\to\infty$ for fixed $q\in[0,1)$.
We call this regime \enquote{regenerative} because of the following property \cite[Corollary~5.1]{PT19} (see also the discussion at the beginning of \cite[Section~3]{BB17}):

\begin{lemma}\label{lem: infinite Mallows is regenerative}
    Let $\Sigma \sim \Mallows{\N^*}{q}$.
    Then we can write $\Sigma = B_1 \oplus \dots \oplus B_k \oplus \dots$ where $(B_k)_{k\ge1}$ is a sequence of i.i.d.~random blocks.
    Conditionally given its size $|B_1|$, the permutation $B_1$ is a random $\Mallows{|B_1|}{q}$ permutation conditioned to be indecomposable.
\end{lemma}

With the notation of the above lemma, we say that $B_1$ is a {\em Mallows block with parameter $q$}, and we denote its distribution by $\MallowsBlock{q}$.
In \cite{H22}, the random permutations $(B_k)_{k\ge1}$ are called the {\em excursions} of $\Sigma$.
The law of $\abs{B_1}$ is described in \cite[Corollary~5.1~(iv)]{PT19} and rather intricate, but according to \cite[Proposition~4.6]{BB17} it admits some exponential moments, and thus finite moments of all orders (see also \cite[Proposition~3.4]{H22}).

As a consequence of \Cref{lem: infinite Mallows is regenerative}, the sub-permutations $\tau_n = \pat{[n]}{\Sigma}$ are in a sense \enquote{almost regenerative}.
Indeed, let $K_n := \max\left\{ k\ge0 : \abs{B_1} +\dots+ \abs{B_k} \le n \right\}$.
Then $\tau_n$ can be written as
\begin{equation}\label{eq: tau as sum of blocks and remainder}
    \tau_n = B_1 \oplus\dots\oplus B_{K_n} \oplus \sigma_n
\end{equation}
for some pattern $\sigma_n$ of $B_{K_n+1}$.
Since the random variable $\abs{B_1}$ is $L^p$ for all $p>0$, it is well-known that $K_n$ satisfies a law of large numbers and a central limit theorem (see e.g.~\cite[Theorems 3.7, 3.8, 3.16]{J18}):

\begin{proposition}\label{th: LLN and CLT for renewal time}
    We have the following convergence as $n\to\infty$:
    \begin{equation*}
        \frac{K_n}{n} 
        \,\overset{a.s.}{\longrightarrow}\, 
        \frac{1}{\expec{\abs{B_1}}}
        \qquad;\qquad
        \frac{K_n - n / \expec{\abs{B_1}}}{n^{1/2}} 
        \,\overset{\text{law}}{\longrightarrow}\, 
        \Normal{0}{\frac{\var{\abs{B_1}}}{\expec{\abs{B_1}}^3}} ,
    \end{equation*}
    where $B\sim\MallowsBlock{q}$.
    Moreover, this holds with convergence of all moments.
\end{proposition}

\subsection{Pattern-joint CLT}\label{sec: proofs pattern-joint CLT q fixed}

This section is devoted to the proofs of \Cref{th: LLN and CLT for joint patterns with fixed q}, \Cref{prop: non-degeneracy of CLT and rank of matrix}, and \Cref{th: CLT for identity pattern with fixed q}.

\begin{proof}[Proof of \Cref{th: LLN and CLT for joint patterns with fixed q}]
    We will use \Cref{eq: tau as sum of blocks and remainder}, and the same techniques as in \cite{J23_forest}.
    Throughout the proof, $B$ denotes a generic $\MallowsBlock{q}$ random variable.

    First, fix a pattern $\pi = \pi_1\oplus\dots\oplus\pi_d\in\S$.
    Observe that any occurrence of a block $\pi_j$ in $\tau_n$ is wholly contained in a single term of the sum $\tau_n = B_1 \oplus\dots\oplus B_{K_n} \oplus \sigma_n$.
    Define $\occinj{\pi}{\tau_n}$ as the number of occurences of $\pi$ in $\tau_n$ where all $\pi_j$'s are contained in distinct $B_k$'s, that is:
    \begin{equation}\label{eq: writing single occinj as asymmetric U-stat}
        \occinj{\pi}{\tau_n}
        := \sum_{1\le k_1<\dots<k_d\le K_n} \prod_{j=1}^d \occ{\pi_j}{B_{k_j}} .
    \end{equation}
    It will be sufficient to prove the desired result for $\occinj{\pi}{\tau_n}$ , after showing that it approximates $\occ{\pi}{\tau_n}$ well.
    For any occurrence of $\pi$ in $\tau_n$ which does not contribute to $\occinj{\pi}{\tau_n}$, either some $\pi_{i}$ appears in the same $B_k$ as $\pi_{i+1}$, or $\pi_d$ appears in $B_{K_{n}+1}$ (recall that $\sigma_n$ is a sub-permutation of $B_{K_n+1}$).
    The first type of contribution is bounded by
    \begin{align*}
        {\rm Occ}^-_1
        \,\le\, \sum_{i=1}^{d-1} \, \sum_{1\le k_j\le K_n ,\, j\ne i} \occ{\pi_{i}}{B_{k_{i+1}}} \prod_{j\ne i} \occ{\pi_j}{B_{k_j}} 
        \,\le\, \sum_{i=1}^{d-1} \, \sum_{1\le k_j\le n ,\, j\ne i} \, \prod_{j\ne i} \abs{B_{k_j}}^{\abs{\pi}} ,
    \end{align*}
    while the second one is bounded by
    \begin{multline*}
        {\rm Occ}^-_2
        \,\le\, \sum_{\ell=1}^d \left( \occ{\pi_\ell \oplus\dots\oplus \pi_d}{B_{K_n+1}} \sum_{1\le k_j\le K_n ,\, j<\ell} \, \prod_{j=1}^{\ell-1} \occ{\pi_j}{B_{k_j}} \right) \\
        \le d\, \abs{B_{K_n+1}}^{\abs{\pi}} \sum_{1\le k_j\le n ,\, j<d} \, \prod_{j=1}^{d-1} \abs{B_{k_j}}^{\abs{\pi}} .
    \end{multline*}
    Recall that $(\abs{B_k})_{k\ge1}$ is a sequence of i.i.d.~random variables with 
    finite moments of all orders \cite[Proposition~4.6]{BB17}.
    Moreover, $\abs{B_{K_n+1}}$ is asymptotically distributed as the size-bias distribution of $\abs{B}$ (see e.g.~\cite[Lemma~3.6]{H22} and the discussion above), thus its moments are all finite and bounded over $n$.
    Therefore:
    \begin{align*}
        \expec{\abs{ \occ{\pi}{\tau_n} \m \occinj{\pi}{\tau_n} }}
        \le \expec{{\rm Occ}^-_1 + {\rm Occ}^-_2} 
        \le \Landau{\O}{}{n^{d-1}} ,
    \end{align*}
    and more generally, for any $p>0$:
    \begin{equation*}
        \expec{\abs{ \occ{\pi}{\tau_n} \m \occinj{\pi}{\tau_n} }^p}
        \le \expec{\left( {\rm Occ}^-_1 + {\rm Occ}^-_2 \right)^p}
        \le \Landau{\O}{}{\left( n^{d-1} \right)^p} .
    \end{equation*}
    Thus it suffices to prove \Cref{th: LLN and CLT for joint patterns with fixed q} with $\occ{\pi}{\tau_n}$ replaced by $\occinj{\pi}{\tau_n}$.

    Now, fix $d\ge1$ and a real vector $\Lambda = \left( \lambda_\pi \right)_{\pi=\pi_1\oplus\dots\oplus\pi_d}$ with finite support.
    Using \eqref{eq: writing single occinj as asymmetric U-stat}, we can write
    \begin{multline*}
        \sum_{\pi = \pi_1\oplus\dots\oplus\pi_d} \lambda_\pi \occinj{\pi}{\tau_n}
        = \sum_{1\le k_1<\dots<k_d \le K_n} \, \sum_{\pi=\pi_1\oplus\dots\pi_d} \lambda_\pi \prod_{j=1}^d \occ{\pi_j}{B_{k_j}}
        \\= \sum_{1\le k_1<\dots<k_d \le K_n} f_\Lambda\left( B_{k_1},\dots,B_{k_d} \right)
    \end{multline*}
    One may thus recognize an asymmetric $U$-statistic stopped at some random stopping time, as studied in \cite{J18}.
    Using \cite[Theorems 3.9 and 3.11]{J18} (beware the notation), we deduce that
    \begin{equation*}
        \frac{1}{n^d} \sum_{\pi=\pi_1 \oplus\dots\oplus \pi_d} \lambda_\pi \occinj{\pi}{\tau_n} 
        \,\cv{}\, 
        \frac{1}{d!} \sum_{\pi=\pi_1 \oplus\dots\oplus \pi_d} \lambda_\pi \prod_{j=1}^d \frac{\expec{\occ{\pi_j}{B}}}{\expec{\abs{B}}}
    \end{equation*}
    almost surely, which concludes the proof of \eqref{eq: LLN q fixed}, and
    \begin{equation}\label{eq: cv distrib of combinations q fixed}
        \frac{ \sum_{\pi=\pi_1 \oplus\dots\oplus \pi_d} \lambda_\pi \occinj{\pi}{\tau_n} - \binom{n}{d} \sum_{\pi=\pi_1 \oplus\dots\oplus \pi_d} \lambda_\pi e_{\pi,q} }{n^{d-1/2}} 
        \,\cv{}\, 
        \Normal{0}{\gamma_\Lambda^2}
    \end{equation}
    in distribution, for some $\gamma_\Lambda\ge0$.
    This also holds with convergence of all moments by \cite[Theorem 3.17]{J18}.
    To describe the asymptotic variance $\gamma_\Lambda^2$, we first need to describe the (non-centered) univariate projections:
    \begin{equation}\label{eq: univariate projections q fixed}
        f_{\Lambda,j}(B) := \expec{ f_\Lambda\left(B_1,\dots,B_d\right) \,\big|\, B_j=B } 
        = \sum_{\pi=\pi_1 \oplus\dots\oplus \pi_d} \lambda_\pi \frac{\occ{\pi_j}{B}}{\expec{\occ{\pi_j}{B}}} \prod_{k=1}^d \expec{\occ{\pi_k}{B}} ,
    \end{equation}
    for $j\in[d]$.
    Note that $\expec{\occ{\pi}{B}}>0$ for all $\pi\in\S$ and $q\in(0,1)$, since the support of $B$ is the set of indecomposable permutations, and any given pattern is contained in at least one indecomposable permutation.
    Using \cite[Equation~(3.28)]{J18}:
    \begin{align}\label{eq: formula mixed variance q fixed}
        \gamma_\Lambda^2
        = &\expec{\abs{B}}^{1-2d}
        \sum_{1\le i,j\le d}
        \binom{i\p j\m 2}{i\m 1} \binom{2d\m i\m j}{d\m i} \frac{1}{(2d - 1)!}
        \cov{f_{\Lambda,j}(B)}{f_{\Lambda,i}(B)} \nonumber
        \\&- \frac{2\expec{\abs{B}}^{-2d}}{(d-1)!d!}
        \left( \sum_{\pi=\pi_1 \oplus\dots\oplus \pi_d} \lambda_\pi \prod_{k=1}^d \expec{\occ{\pi_k}{B}}\right)
        \sum_{i=1}^d \cov{\abs{B}}{f_{\Lambda,i}(B)} \nonumber
        \\&+ \frac{\expec{\abs{B}}^{-2d-1}}{(d-1)!^2}
        \left( \sum_{\pi=\pi_1 \oplus\dots\oplus \pi_d} \lambda_\pi \prod_{k=1}^d \expec{\occ{\pi_k}{B}} \right)^2
        \var{\abs{B}} .
    \end{align}
    Substituting in \eqref{eq: univariate projections q fixed} and developing, we can write $\gamma_\Lambda^2 = \Lambda^{\rm T} \Gamma^{(d)} \Lambda$ where, for all patterns $\pi=\pi_1 \oplus\dots\oplus \pi_d$ and $\rho=\rho_1 \oplus\dots\oplus \rho_d$:
    \begin{align}\label{eq: formula covariance matrix q fixed}
        \Gamma_{\pi,\rho}^{(d)}
        = \expec{\abs{B}} e_{\pi,q} e_{\rho,q}\Bigg(
        &\frac{1}{(2d-1)!}
        \sum_{1\le i,j\le d}
        \binom{i\p j\m 2}{i\m 1} \binom{2d\m i\m j}{d\m i} \frac{\cov{\occ{\pi_j}{B}}{\occ{\rho_i}{B}}}{\expec{\occ{\pi_j}{B}} \expec{\occ{\rho_i}{B}}} \nonumber
        \\&- \frac{1}{(d-1)!d!}
        \sum_{i=1}^d \left( \frac{\cov{\abs{B}}{\occ{\rho_i}{B}}}{\expec{\abs{B}} \expec{\occ{\rho_i}{B}}} + \frac{\cov{\abs{B}}{\occ{\pi_i}{B}}}{\expec{\abs{B}} \expec{\occ{\pi_i}{B}}} \right) \nonumber
        \\&+ \frac{1}{(d-1)!^2}
        \frac{\var{\abs{B}}}{\expec{\abs{B}}^2} \Bigg) .
    \end{align}
    \Cref{eq: CLT joint q fixed} then follows from \Cref{eq: cv distrib of combinations q fixed} along with the Cram\'er--Wold theorem.

    Finally, to prove convergence of all joint moments in \eqref{eq: CLT joint q fixed}, by uniform integrability it suffices to argue that all joint moments are bounded as $n\to\infty$.
    This simply follows from H\"older's inequality, and the fact that \eqref{eq: cv distrib of combinations q fixed} holds with convergence of all moments.
\end{proof}

\begin{remark}\label{rem: explicit formula for e with single inversion}
    For $\pi=2\,1$, it is known that $\expec{\inv{\tau_n}} \sim \frac{q}{1-q} n$ as $n\to\infty$ \cite[Proposition~29]{Bevan_Threlfall_2024}.
    This shows the equality 
    \[
    e_{2\,1,q} = \frac{\expec{\inv{B}}}{\expec{\abs{B}}} = \frac{q}{1-q} \,,
    \]
    for which we do not have a direct proof.
    Since $e_{\pi,q} = \prod_j e_{\pi_j,q}$, this also implies explicit formulas for patterns whose blocks are only points or inversions, e.g.~$e_{1\,3\,2 , q} = e_{2\,1\,3 , q} = \frac{q}{1-q}$.
    For other patterns $\pi$, computing $e_{\pi,q}$ seems difficult.
\end{remark}

Before computing the ranks of specific sub-matrices of $\Gamma^{(d)}$, we state and prove a simple lemma which will be useful in the proof of \Cref{prop: non-degeneracy of CLT and rank of matrix}.

\begin{lemma}\label{lem: rank system triangular}
    Fix an integer $d\ge2$.
    Consider the linear system
    \begin{equation}\label{eq: linear system triangular}
        \forall j\in[d],\quad
        \sum_{i<j} \lambda_{i,j} + \sum_{k>j} \lambda_{j,k} = 0 \tag{E}
    \end{equation}
    in the real variables $\left( \lambda_{i,k} \right)_{1\le i< k\le d}$.
    
    If $d=2$ then this system has rank $1$, and if $d\ge3$ then it has rank $d$.
\end{lemma}

\begin{proof}
    The linear system \eqref{eq: linear system triangular} has $d$ equations and $\binom{d}{2}$ unknowns, thus we aim to prove that it has full rank.
    If $d=2$, this is immediate.
    If $d\ge3$, it suffices to find an invertible sub-system of size $d\times d$.
    The system \eqref{eq: linear system triangular} restricted to the variables $\lambda_{1,2}, \lambda_{1,3}, \dots, \lambda_{1,d}$ and $\lambda_{d-1,d}$ rewrites as:
    \begin{equation*}
        \sum_{k>1} \lambda_{1,k} = 0
        \;;\quad
        \lambda_{1,2} = 0
        \;;\quad
        \lambda_{1,3} = 0
        \;;\;\dots\;;\;
        \lambda_{1,d-2} = 0
        \;;\quad
        \lambda_{1,d-1} + \lambda_{d-1,d} = 0
        \;;\quad
        \lambda_{1,d} + \lambda_{d-1,d} = 0 ,
    \end{equation*}
    which is easily seen to be invertible.
    This concludes the proof.
\end{proof}

\begin{proof}[Proof of \Cref{prop: non-degeneracy of CLT and rank of matrix}]
    First, fix a pattern $\pi=\pi_1 \oplus\dots\oplus \pi_d$.
    As in the proof of \Cref{th: LLN and CLT for joint patterns with fixed q}, we can use the results of \cite{J18}.
    In this case, the real vector is set to $\Lambda := \left( \One{\rho=\pi} \right)_{\rho=\rho_1 \oplus\dots\oplus \rho_d}$ and the uncentered univariate projections of \eqref{eq: univariate projections q fixed} are given by $f_{\Lambda,j}(B) = \frac{\occ{\pi_j}{B}}{\expec{\occ{\pi_j}{B}}} \prod_{k=1}^d \expec{\occ{\pi_k}{B}}$ for $j\in[d]$.
    Therefore by \cite[Equation~(3.29)]{J18}, $\Gamma_{\pi,\pi}^{(d)} = 0$ if and only if for all $j\in[d]$:
    \begin{equation*}
        \expec{\abs{B}} 
        \left( \frac{\occ{\pi_j}{B}}{\expec{\occ{\pi_j}{B}}} -1 \right) \prod_{k=1}^d \expec{\occ{\pi_k}{B}} 
        = \left(\big. \abs{B} - \expec{\abs{B}} \right) 
        \prod_{k=1}^d \expec{\occ{\pi_k}{B}} 
    \end{equation*}
    almost surely.
    Since $\prod_{k=1}^d \expec{\occ{\pi_k}{B}} >0$, this happens if and only if for all $j\in[d]$:
    \begin{equation*}
        \frac{\occ{\pi_j}{B}}{\expec{\occ{\pi_j}{B}}} = \frac{\abs{B}}{\expec{\abs{B}}} 
    \end{equation*}
    almost surely.
    If $\pi$ is not the identity permutation, then at least one of its blocks $\pi_j$ has size at least $2$.
    Since $\prob{\abs{B}=1} = \prob{\Sigma(1)=1} = 1-q >0$, the event $\left\{ \occ{\pi_j}{B}=0 \right\}$ then occurs with positive probability.
    Since $\prob{B=\pi_j} >0$, the event $\left\{ \occ{\pi_j}{B}>0 \right\}$ also occurs with positive probability.
    Hence $\Gamma_{\pi,\pi}^{(d)} > 0$ in this case.
    On the other hand, if $\pi = \id_d$ then $\abs{\pi_j}=1$ for all $j\in[d]$, therefore $\occ{\pi_j}{B} = \abs{B}$ a.s.~and this implies $\Gamma_{\id,\id}^{(d)} = 0$.
    This concludes the proof of the first claim, about (non-)degeneracy of the marginal CLT.

    For any $r\ge d$, we can compute the rank of the finite matrix $\Gamma^{(d,r)}$ by describing its kernel.
    Fix a real vector $\Lambda = \left( \lambda_\pi \right)_{\pi = \pi_1 \oplus\dots\oplus \pi_d \in\S_r}$.
    Then $\Lambda^{\rm T}\Gamma^{(d,r)}\Lambda = \gamma_\Lambda^2$ is given by \Cref{eq: formula mixed variance q fixed}.
    As before, by \cite[Equation~(3.29)]{J18}, $\gamma_\Lambda = 0$ if and only if for all $j\in[d]$:
    \begin{equation}\label{eq: CNS kernel covariance matrix}
        f_{\Lambda,j}(B)
        = \abs{B} \,
        \frac{\expec{f_{\Lambda,j}(B)}}{\expec{\abs{B}}}
    \end{equation}
    almost surely.
    Before studying this relation in more detail, let us rewrite the univariate projection $f_{\Lambda,j}$ of \eqref{eq: univariate projections q fixed} as:
    \begin{align*}
        f_{\Lambda,j}(B)
        &= \sum_{\pi=\pi_1 \oplus\dots\oplus \pi_d \in\S_r} \lambda_\pi
        \frac{\occ{\pi_j}{B}}{\expec{\occ{\pi_j}{B}}} \prod_{k=1}^d \expec{\occ{\pi_k}{B}} 
        \\&= \sum_{\rho\in\S} \, \sum_{\pi=\pi_1 \oplus\dots\oplus \pi_d \in\S_r ,\, \pi_j=\rho} \lambda_\pi
        \frac{\occ{\rho}{B}}{\expec{\occ{\rho}{B}}} \prod_{k=1}^d \expec{\occ{\pi_k}{B}} 
        \\&= \sum_{\rho\in\S} \lambda_{\rho,j} \, \occ{\rho}{B} 
    \end{align*}
    where, for any $j\in[d]$ and $\rho\in\S$:
    \begin{equation}\label{eq: def lambda_rho_j for kernel}
        \lambda_{\rho,j}
        = \sum_{\pi=\pi_1 \oplus\dots\oplus \pi_d \in\S_r ,\, \pi_j=\rho} \lambda_\pi
        \prod_{k\ne j} \expec{\occ{\pi_k}{B}} .
    \end{equation}
    Note that $\lambda_{\rho,j} = 0$ whenever $\abs{\rho} > r-d+1$.
    Let us consider two complementary cases:
    suppose that $\lambda_{\rho,j}=0$ for all $j\in[d]$ and all $\rho\in\S$ such that $\abs{\rho}\ge2$.
    Then $f_{\Lambda,j}(B) = \lambda_{1,j} \abs{B}$ a.s.~for all $j\in[d]$, thus \eqref{eq: CNS kernel covariance matrix} is satisfied and $\gamma_\Lambda = 0$.
    On the other hand, suppose that $\lambda_{\rho_0,j}\ne0$ for some $j\in[d]$ and some $\rho_0\in\S$ such that $\abs{\rho_0}\ge2$, and take such a $\rho_0$ with minimum size.
    The event $\left\{ B=1 ,\; f_{\Lambda,j}(B) = \lambda_{1,j} \abs{B} \right\}$ has positive probability, and the event $\left\{ B=\rho_0 ,\; f_{\Lambda,j}(B) = \lambda_{1,j} \abs{B} + \lambda_{\rho_0,j} \right\}$ as well.
    Hence the ratio $f_{\Lambda,j}(B) / \abs{B}$ is not constant a.s., so \eqref{eq: CNS kernel covariance matrix} does not hold and we deduce $\gamma_\Lambda > 0$.
    
    The above discussion shows that $\gamma_\Lambda = 0$ if and only if $\lambda_{\rho,j}=0$ for all $j\in[d]$ and all $\rho\in\S$ such that $\abs{\rho}\ge2$, where $\lambda_{\rho,j}$ is defined by \eqref{eq: def lambda_rho_j for kernel}.
    We can use this criterium to study the rank of $\Gamma^{(d,r)}$.
    \begin{itemize}
        \item If $d=r$, then the only pattern in $\S_r$ wih $d$ blocks is the identity permutation. 
        It has already been established that $\Gamma^{(r,r)}_{\id,\id} = 0$.
        \item If $d=r-1$, then there are exactly $r-1$ patterns in $\S_r$ with $d$ blocks: for each $j\in[r-1]$, the adjacent transposition $\pi^{(j)}$ which swaps $j$ with $j+1$.
        The only non-trivial block of $\pi^{(j)} = \pi^{(j)}_1 \oplus\dots\oplus \pi^{(j)}_d$ is $\pi^{(j)}_j = 2\,1$.
        Therefore we simply have $\lambda_{2 1 , j} = \lambda_{\pi^{(j)}} \expec{\abs{B}}^{r-2}$ for all $j\in[d]$.
        By the above discussion, $\Lambda^{\rm T}\Gamma^{(r-1,r)}\Lambda = 0$ if and only if $\lambda_{\pi^{(j)}} = 0$ for all $j\in[d]$, that is if $\Lambda$ is the null vector.
        Hence $\Gamma^{(r-1,r)}$ is invertible.
        \item The case $d=1$ corresponds to indecomposable permutations in $\S_r$.
        \Cref{eq: def lambda_rho_j for kernel} then becomes trivial:
        for each indecomposable $\pi\in\S_r$, $\lambda_{\pi,1} = \lambda_\pi$.
        Therefore $\Lambda^{\rm T}\Gamma^{(1,r)}\Lambda = 0$ if and only if $\Lambda$ is the null vector, and thus $\Gamma^{(1,r)}$ is invertible.
        \item If $d=r-2 \ge2$, then there are two types of patterns in $\S_r$ with $d$ blocks:
        for each $1\le i<j \le d$, the bi-adjacent transposition $\pi^{(i,j)}$ which swaps $i$ with $i+1$ and $j+1$ with $j+2$, whose only non-trivial blocks are $\pi^{(i,j)}_i = \pi^{(i,j)}_j = 2\,1$;
        and for each $j\in[d]$ and $\rho\in \{ 231,312,321 \}$, the permutation $\pi^{(\rho,j)}$ whose only non-trivial block is $\pi^{(\rho,j)}_j = \rho$.
        Thus for each $j\in[d]$:
        \begin{equation*}
            \left\{
            \begin{array}{ll}
                \lambda_{\rho,j} = \lambda_{\pi^{(\rho,j)}} \expec{\abs{B}}^{r-3} & \text{for all } \rho\in\{231,312,321\} ;\smallskip\\
                \lambda_{21,j} = \sum_{1\le i< j} \lambda_{\pi^{(i,j)}} \expec{\abs{B}}^{r-4} + \sum_{j< k\le d} \lambda_{\pi^{(j,k)}} \expec{\abs{B}}^{r-4} .
            \end{array}
            \right.
        \end{equation*}
        The equation $\Lambda^{\rm T}\Gamma^{(r-2,r)}\Lambda = 0$ is then equivalent to the following for all $j\in[d]$:
        \begin{equation*}
            \left\{
            \begin{array}{ll}
                \lambda_{\pi^{(\rho,j)}} = 0 & \text{for all } \rho\in\{231,312,321\} ;\smallskip\\
                \sum_{1\le i< j} \lambda_{\pi^{(i,j)}}  + \sum_{j< k\le d} \lambda_{\pi^{(j,k)}} = 0 .
            \end{array}
            \right.
        \end{equation*}
        If this holds then the variables $\left( \lambda_{\pi^{(i,j)}} \right)_{1\le i<j\le d}$ solve the linear system \eqref{eq: linear system triangular}.
        By \Cref{lem: rank system triangular}, the solution space of $\Lambda^{\rm T}\Gamma^{(r-2,r)}\Lambda = 0$ has dimension $\binom{d}{2} - d$ if $d\ge3$, and dimension $\binom{d}{2} - 1 = 0$ if $d=2$.

        There are $\binom{d}{2} + 3d$ patterns in $\S_r$ with $d=r-2$ blocks.
        If $d=2$ then the matrix $\Gamma^{d,r} = \Gamma^{2,4}$ is invertible, and has full rank $7$.
        If $d\ge3$ then the matrix $\Gamma^{d,r}$ has rank $\binom{d}{2} + 3d - \binom{d}{2} + d = 4d$; it is invertible only if $d=3$.
    \end{itemize}
    This discussion concludes the proof of \Cref{prop: non-degeneracy of CLT and rank of matrix}.
\end{proof}

\begin{proof}[Proof of \Cref{th: CLT for identity pattern with fixed q}]
    We have the following trivial identity:
    \begin{equation*}
        \sum_{\pi\in\S_r} \occ{\pi}{\tau_n} = \binom{n}{r} .
    \end{equation*}
    We can distinguish between three types of patterns in $\S_r$: the identity permutation; the adjacent transpositions (which have $r-1$ blocks); and the patterns which have at most $r-2$ blocks.
    The contribution of the latter to the above sum is bounded by \Cref{th: LLN and CLT for joint patterns with fixed q}, \Cref{eq: LLN q fixed}:
    \begin{equation*}
        \sum_{\pi = \pi_1 \oplus\dots\oplus \pi_d \in\S_r ,\, d\le r-2} \occ{\pi}{\tau_n} \le \Landau{\O_\P}{}{n^{r-2}} ,
    \end{equation*}
    whence:
    \begin{equation}\label{eq: decomposition increasing subsequences}
        \occ{\id_r}{\tau_n} = \binom{n}{r} - \sum_{\pi = \pi_1 \oplus\dots\oplus \pi_{r-1} \in\S_r} \occ{\pi}{\tau_n} - \Landau{\O_\P}{}{n^{r-2}} .
    \end{equation}
    The contribution of adjacent transpositions is given by the joint convergence in distribution of \Cref{th: LLN and CLT for joint patterns with fixed q}:
    \begin{equation}\label{eq: CLT sum adjacent transpositions}
        \frac{ \sum_{\pi = \pi_1 \oplus\dots\oplus \pi_{r-1} \in\S_r} \occ{\pi}{\tau_n} - \binom{n}{r-1} \sum_{\pi = \pi_1 \oplus\dots\oplus \pi_{r-1} \in\S_r} \prod_{j=1}^{r-1} \frac{\expec{\occ{\pi_j}{B}}}{\expec{\abs{B}}} }{n^{r-3/2}}
        \,\cv{n\to\infty}\,
        \Normal{0}{\Lambda^{\rm T}\Gamma^{(r-1)}\Lambda} ,
    \end{equation}
    where $\Lambda = \left( \One{\abs{\pi}=r} \right)_{\pi_1 \oplus\dots\oplus \pi_{r-1} \in\S}$.
    Since adjacent transpositions contain $r-2$ trivial blocks and one $2\,1$ block, the asymptotic mean simplifies as:
    \begin{equation*}
        \sum_{\pi = \pi_1 \oplus\dots\oplus \pi_{r-1} \in\S_r} \prod_{j=1}^{r-1} \frac{\expec{\occ{\pi_j}{B}}}{\expec{\abs{B}}}
        = (r-1) \frac{\expec{\inv{B}}}{\expec{\abs{B}}} .
    \end{equation*}
    It remains to compute the asymptotic variance $\Lambda^{\rm T} \Gamma^{(r-1)} \Lambda =: \gamma_r^2$.
    It is positive, since the matrix $\Gamma^{(r-1,r)}$ is invertible by \Cref{prop: non-degeneracy of CLT and rank of matrix}.
    Furthermore it can be expressed via \eqref{eq: univariate projections q fixed} and \eqref{eq: formula mixed variance q fixed}:
    \begin{equation*}
        f_{\Lambda,j}(B) 
        = \inv{B} \, \expec{\abs{B}}^{d-1} 
        + (d-1) \, \abs{B} \, \expec{\abs{B}}^{d-2} \, \expec{\inv{B}} 
    \end{equation*}
    for all $j\in[d]$, with $d:=r-1$.
    Subsequently:
    \begin{align*}
        \gamma_r^2
        = &\expec{\abs{B}}^{-3}
        \sum_{1\le i,j\le d}
        \binom{i\p j\m 2}{i\m 1} \binom{2d\m i\m j}{d\m i} \frac{1}{(2d - 1)!}
        \var{\big. \inv{B} \expec{\abs{B}} + (d-1) \abs{B} \expec{\inv{B}}} 
        \\&- \frac{2\expec{\abs{B}}^{-3}}{(d-1)!d!}
        d^2 \expec{\inv{B}} 
        \cov{\big. \abs{B}}{\inv{B} \expec{\abs{B}} + (d-1) \abs{B} \expec{\inv{B}}} 
        \\&+ \frac{\expec{\abs{B}}^{-3}}{(d-1)!^2}
        d^2 \expec{\inv{B}}^2
        \var{\abs{B}} .
    \end{align*}
    To simplify this expression we can use the formula $\sum_{1\le i,j\le d} \binom{i+ j- 2}{i- 1} \binom{2d- i- j}{d- i} = \frac{(2d-1)!}{(d-1)!^2}$, which can 
    e.g.~be obtained via Beta integrals, or combinatorially\footnote{
    Indeed, rewrite it as
    $$ \sum_{1\le i,j\le d} \binom{i+ j- 2}{i- 1} \binom{2d- i- j}{d- i} 
    = \sum_{i=1}^d \sum_{k=i}^{i+d-1} \binom{k-1}{i- 1} \binom{2d-1-k}{d-i} 
    = d \binom{2d-1}{d}
    = \frac{(2d-1)!}{(d-1)!^2} ,$$
    where the middle equality comes from counting pointed subsets $E$ of $[2d-1]$ with $d$ elements: $k$ is the pointed element of $E$, and $i-1$ is the number of elements in $E$ lower than $k$.
    }.
    Therefore:
    \begin{equation*}
        \gamma_r^2 
        = \expec{\abs{B}}^{-3} \left(\big.
        c_1 \var{\abs{B}}
        + c_2 \cov{\abs{B}}{\inv{B}}
        + c_3 \var{\inv{B}}
        \right)
    \end{equation*}
    where
    \begin{equation*}
        \left\{
        \begin{array}{lll}
        c_1 = \left(
        \frac{(d-1)^2}{(d-1)!^2}
        - \frac{2 d^2 (d-1)}{(d-1)!d!}
        + \frac{d^2}{(d-1)!^2} 
        \right) \expec{\inv{B}}^2 
        = \frac{1}{(d-1)!^2} \expec{\inv{B}}^2 ;\\
        c_2 = \left(
        \frac{2 (d-1)}{(d-1)!^2} 
        - \frac{2 d^2 }{(d-1)!d!}
        \right) \expec{\inv{B}} \expec{\abs{B}}
        = \frac{-2}{(d-1)!^2} \expec{\inv{B}} \expec{\abs{B}} ;\\
        c_3 = \frac{1}{(d-1)!^2} \expec{\abs{B}}^2 .
        \end{array}
        \right.
    \end{equation*}
    Finally, we can write
    \begin{equation}\label{eq: asymptotic variance increasing subsequences}
        \gamma_r^2 = \frac{\expec{\abs{B}}^{-3}}{(d-1)!^2} \var{\big. \expec{\inv{B}} \abs{B} - \expec{\abs{B}} \inv{B} }
    \end{equation}
    with $d=r-1$.
    Convergence of moments also holds as usual, and this concludes the proof.
\end{proof}

\section{Patterns in the continuous-time infinite Mallows process}

\subsection{The continuous-time infinite Mallows process}\label{sec: CTI Mallows}

In order to define a continuous-time coupling of $\Sigma_t \sim \Mallows{\N^*}{t}$ for $t\in[0,1)$, we use continuous-time geometric processes along with \Cref{lem: range of psi_infty}.
In \cite[Section~3.2]{AK24} this is done for the {\em bi-infinite} Mallows distribution (on $\Z$), which is trickier to construct than the \enquote{unilateral} version (on $\N^*$) considered here.

Say that a stochastic process $X:\R_+\to\N\cup\{\infty\}$ is a {\em birth process} if it is càdlàg, time-inhomogeneous Markov, and admits {\em infinitesimal birth rates} $p(j,t)\ge0$ such that:
\begin{equation}\label{eq: def birth rates}
    \text{for all } t\ge0 \text{ and } j\in\N, \quad
    \left\{
    \begin{array}{ll}
        \prob{ X_{t+h} = j \,\big|\, X_t=j} = 1 - h p(j,t) + o(h) ;\smallskip\\
        \prob{X_{t+h} = j\p1 \,\big|\, X_t=j} = h p(j,t) + o(h) ;
    \end{array}
    \right.
    \text{ as } h\to0.
\end{equation}
The law of $X$ is then uniquely determined by these birth rates along with the law of $X_0$.
Conversely, these birth rates are uniquely determined by the functions $t\mapsto\prob{X_t=j}$, $j\in\N$.
This follows from the Chapman--Kolmogorov forward equation (see e.g.~the proof of \cite[Proposition~3.5]{AK24}).

\begin{lemma}[\cite{AK24}]\label{lem: geometric birth process}
    Let $R:[0,1)\to \N^*$ be the birth process with $R_0=1$ and birth rates $p(j,t) = \frac{j}{1-t}$ for $j\in\N^*$, $t\in[0,1)$.
    Then $R$ is non-explosive (i.e.~it is a.s.~well-defined on $[0,1)$, and has finitely many jumps on any proper subinterval), and has marginal distributions $R_t \sim \Geom{1\m t}$ for all $t\in[0,1)$.
    
    Conversely, if $R:[0,1)\to \N^*$ is a birth process with marginal distributions $R_t \sim \Geom{1\m t}$ for all $t\in[0,1)$, then it has birth rates $p(j,t) = \frac{j}{1-t}$ for $j\in\N^*$, $t\in[0,1)$.
\end{lemma}

The process $R$ of \Cref{lem: geometric birth process} shall be refered to as the {\em geometric birth process}.
We can now state an infinite analog of \cite[Theorem~1.1]{C22}.
Note that the authors of \cite{C22} and \cite{AK24} chose to work with {\em left}-inversion counts, but for the purpose of this paper, {\em right}-inversion counts seem better adapted.
In particular, we do not know if an analog of \Cref{th: CTI Mallows process as sum of blocks} would hold with left-inversions.

\begin{theorem}\label{th: continuous-time infinite Mallows process}
    There exists a unique time-inhomogeneous Markov process $\Sigma :[0,1) \to \S_\infty$ such that:
    \begin{enumerate}
        \item\label{item: marginals of infinite continous Mallows} for all $t\in[0,1)$, $\Sigma_t$ has marginal distribution $\Mallows{\N^*}{t}$;
        \item\label{item: right-inversions of infinite continous Mallows} the right-inversion counts $r_i(\Sigma) : [0,1)\to\N$, $i\ge1$ are independent birth processes.
    \end{enumerate}
    Furthermore, we can write $r_i(\Sigma_t) = R_{i,t}\m1$ for all $i\ge1$ and $t\in[0,1)$, where $\left( R_{1,t} \right)_{t\in[0,1)}, \left( R_{2,t} \right)_{t\in[0,1)}, \dots$ are i.i.d.~geometric birth processes as defined in \Cref{lem: geometric birth process}.
\end{theorem}

We call the process $\Sigma = \left( \Sigma_t \right)_{t\in[0,1)}$ of \Cref{th: continuous-time infinite Mallows process} the {\em continuous-time infinite Mallows process}, and write $\ContinuousMallows{\N^*}$ for its law.

\begin{proof}
    Let $\left( R_{i,\cdot} \right)_{i\ge1}$ be i.i.d.~geometric birth process.
    By the Borel--Cantelli lemma, for any fixed $t\in[0,1)$, almost surely $R_{i,t}=1$ for infinitely many $i$'s.
    Since the processes $\left( R_{i,t} \right)_{t\in[0,1)}$ are monotone, almost surely, this holds for all $t\in[0,1)$.
    Therefore by \Cref{lem: range of psi_infty}, almost surely, we can define $\Sigma_t \in \S_\infty$ with right-inversion counts $r_i\left( \Sigma_t \right) := R_{i,t}-1$ for all $t\in[0,1)$.

    For each $t\in[0,1)$, $\Sigma_t \sim \Mallows{\N^*}{t}$ by definition.
    Since there is a one-to-one correspondence $\psi_\infty$ between infinite permutations and admissible sequences of right-inversion counts, the Markov property of $\left(\Sigma_t\right)_{t\in[0,1)}$ follows from that of the independent processes $\left(R_{i,t}\right)_{t\in[0,1)}$, $i\ge1$.
    Hence this proves the existence of continuous-time infinite Mallows processes.

    Now, let $\Sigma'$ be a time-inhomogeneous Markov process of infinite permutations, satisfying Properties~\ref{item: marginals of infinite continous Mallows} and~\ref{item: right-inversions of infinite continous Mallows}.
    Then for any fixed $t\in[0,1)$, by Property~\ref{item: marginals of infinite continous Mallows}, the sequence $\left( R_{i,t}' \right)_{i\ge1} := \left( r_i(\Sigma_t') +1 \right)_{i\ge1}$ is a sequence of i.i.d.~$\Geom{1-t}$ r.v.'s.
    Thus by Property~\ref{item: right-inversions of infinite continous Mallows}, the processes $\left( R_{i,t}' \right)_{t\in[0,1)}$, $i\ge1$, are independent birth process with geometric marginal laws; therefore, by \Cref{lem: geometric birth process}, they are i.i.d.~geometric birth processes.
    Using the one-to-one correspondence $\psi_\infty$, we see that the process $\left( \Sigma_t' \right)_{t\in[0,1)}$ has the same distribution as $\left( \Sigma_t \right)_{t\in[0,1)}$.
\end{proof}

When studying the infinite Mallows permutation $\Sigma \sim \Mallows{\N^*}{q}$ for fixed $q\in[0,1)$, it is useful to decompose $\Sigma = B_1 \oplus B_2 \oplus \dots$ as an infinite sum of i.i.d.~blocks.
Analogously, in order to study the continuous-time infinite Mallows process $\left( \Sigma_t \right)_{t\in[0,1)} \sim \ContinuousMallows{\N^*}$, we introduce i.i.d.~\enquote{block processes} such that $\Sigma_t = B_{1,t} \oplus B_{2,t} \oplus \dots$ for all $t\in[0,1)$.
To obtain the i.i.d.~property, we require the sizes $\abs{B_{i,t}}$ to be constant over $t$.
The fact that $\Sigma_t(1) \cv{t\to1^-} \infty$ a.s.~then forces us to restrict our processes to an arbitrary sub-interval $[0,q] \subset [0,1)$.

For any fixed $q \in [0,1)$, we define the law $\ContinuousBlock{q}$ as follows.
First, let $B^{q} \sim \MallowsBlock{q}$ and set $M := \abs{B^{q}}$.
Then, conditionally on $M$, let $\left( R_{1,t} \right)_{t\in[0,q]}, \dots, \left( R_{M,t} \right)_{t\in[0,q]}$ be independent geometric birth processes conditioned to satisfy $R_{i,q} = r_i\left( B^{q} \right) +1$ for all $i\in[M]$.
Finally, for all $t\in[0,q]$, define $B^{q}_t \in \S_M$ as the permutation with right-inversion counts $r_i\left( B^{q}_t \right) := R_{i,t} -1$.
This is well-defined by \Cref{lem: ranges of phi_n and psi_n}, since $R_{i,t}-1 \le R_{i,q}-1 \le M-i$ for all $t\le q$.
We say that $(B_t^{q})_{t\in[0,q]}$ is a {\em $q$-block process}, and denote by $\ContinuousBlock{q}$ its law.

Note that $B^{q}_{q} \sim \MallowsBlock{q}$ is a.s.~an indecomposable permutation, but this might not be the case for all $B^{q}_t$, $t\in[0,q)$.
For instance, $B_0^q$ is a.s.~an identity permutation.
It is easy to check with \Cref{lem: CNS right-inversion decomposable finite} that if $B^{q}_t$ is not indecomposable for some $t$, then neither is $B^{q}_{t'}$ for all $t'\le t$.

\begin{theorem}\label{th: CTI Mallows process as sum of blocks}
    Let $\left( \Sigma_t \right)_{t\in[0,1)} \sim \ContinuousMallows{\N^*}$.
    Then, for any fixed $q\in[0,1)$, we can write 
    $$\Sigma_t = B^{q}_{1,t} \oplus \dots \oplus B^{q}_{k,t} \oplus \dots$$ 
    for all $t\in[0,q]$, where $\left( B^{q}_{k,t} \right)_{t\in[0,q]}$, $k\ge1$, are i.i.d.~$q$-block processes.
\end{theorem}

\begin{proof}
    By \Cref{lem: infinite Mallows is regenerative}, we can write $\Sigma_{q} = B^q_{1} \oplus\dots\oplus B^q_{k} \oplus\dots$ where $B^q_{1}, \dots, B^q_{k}, \dots$ are i.i.d.~random blocks with $\MallowsBlock{q}$ distribution.
    By \Cref{lem: CNS right-inversion decomposable infinite}, since the right-inversion counts of $\left( \Sigma_t \right)_{t\in[0,q]}$ are non-decreasing as $t$ goes from $0$ to $q$, we can write $\Sigma_{t} = B^q_{1,t} \oplus\dots\oplus B^q_{k,t} \oplus\dots$ for all $t\in[0,q]$, where $B^q_{k,t}$ has the same size as $B^q_{k}$ for all $k\ge1$ (but might not be indecomposable). 

    For all $i\ge1$ and $t\in[0,q]$, let $R_{i,t} := r_i\left( \Sigma_t \right) +1$.
    By \Cref{th: continuous-time infinite Mallows process}, the processes $\left( R_{i,t} \right)_{t\in[0,q]}$ are i.i.d.~geometric birth processes.
    Write $M_k := \abs{B^q_{k,t}}$ for all $k\ge1$.
    An important observation is that the cumulative sums $M_1 +\dots+ M_k$ are stopping times for the i.i.d.~sequence of processes $\left( R_{i,\cdot} \right)_{i\ge1}$.
    Indeed, $M_1 +\dots+ M_k$ is the first index $m$ for which the sequence $R_{1,q}, \dots, R_{m,q}$ defines a sum of $k$ blocks.
    Furthermore, the sequence $\left( R_{i,\cdot} \right)_{i\le M_1}$ is independent from the sequence $\left( R_{i,\cdot} \right)_{i>M_1}$, and the latter is distributed like $\left( R_{i,\cdot} \right)_{i\ge1}$.
    As a consequence, the permutation processes $\left( B^q_{k,t} \right)_{t\in[0,q]}$, $k\ge1$ are i.i.d., and it remains to check that they are distributed like $q$-block processes.
    
    Conditionally given $\Sigma_{q}$, the processes $\left( R_{i,t} \right)_{t\in[0,q]}$, $i\ge1$, are independent geometric birth processes conditioned on their end-values $R_{i,q} = r_i\left( \Sigma_{q} \right) +1$.
    Thus $\left( R_{i,t} \right)_{t\in[0,q]}$, $i\in[M_1]$, are independent geometric birth processes conditioned to satisfy $R_{i,q} = r_i\left( B^q_{1} \right) +1$.
    Hence the permutation process $\left( B^q_{1,t} \right)_{t\in[0,q]}$, which has right-inversion counts $r_i\left( B^q_{1,t} \right) = R_{i,t} -1$ for $i\in[M_1]$, is indeed a $q$-block process.
\end{proof}

\subsection{Time-joint CLT and limit Gaussian process}\label{sec: proof time-joint CLT}

We construct a continuous-time coupling of finite Mallows permutations as follows:
let $\left( \Sigma_t \right)_{t\in[0,1)}$ be a continuous-time infinite Mallows process as defined in \Cref{th: continuous-time infinite Mallows process}, then let $\tau_{n,t} := \pat{[n]}{\Sigma_t}$ for all $n\in\N^*$ and $t\in[0,1)$.

For any fixed $n\in\N^*$, the process $\left( \tau_{n,t} \right)_{t\in[0,1)}$ has $\Mallows{n}{t}$ marginals by \Cref{lem: pattern of infinite Mallows}, and its right-inversion counts have unitary jumps.
Using the terminology of \cite{C22}, it is therefore a smooth Mallows process.
However it is not clear whether its right-inversion counts are independent and Markov, i.e.~whether $\left( \tau_{n,t} \right)_{t\in[0,1)]}$ is a regular Markov process (with respect to its {\em right}-inversions rather than its {\em left}-inversions), or even \enquote{close} to being one.

\begin{proof}[Proof of \Cref{th: Donsker for time-joint single pattern}]
    Arbitrarily fix $q\in(0,1)$.
    By \Cref{th: CTI Mallows process as sum of blocks}, we can write $\left( \Sigma_t \right)_{t\in[0,q]}$ as a sum of i.i.d.~$q$-block processes:
    $$ \Sigma_t = B^q_{1,t} \oplus\dots\oplus B^q_{k,t} \oplus\dots $$
    for all $t\in[0,q]$.
    Therefore 
    $$ \tau_{n,t} = B^q_{1,t} \oplus\dots\oplus B^q_{K_n,t} \oplus \sigma_{n,t} $$ 
    for all $t\in[0,q]$, where $K_n := \max\left\{ k\ge0 : \abs{B^q_{1}} +\dots+ \abs{B^q_{k}} \le n \right\}$ and $\sigma_{n,t}$ is a pattern of $B^q_{K_n+1,t}$.
    Note that $K_n$ depends on $q$, but not on $t$.
    As in the proof of \Cref{th: LLN and CLT for joint patterns with fixed q}, we can define
    \begin{equation*}
        \occinj{\pi}{\tau_{n,t}} := \sum_{1\le k_1<\dots<k_d \le K_n} \prod_{j=1}^d \occ{\pi_j}{B^q_{k_j,t}} ,
    \end{equation*}
    and for any $p>0$ we have
    \begin{equation}\label{eq: approx occ occinj in Donsker}
        \expec{ \abs{ \occ{\pi}{\tau_{n,t}} - \occinj{\pi}{\tau_{n,t}} }^p } \le \Landau{\O}{}{ \left( n^{d-1} \right)^p }
    \end{equation}
    for all $t\in[0,q]$, where the constant inside the $\O$ depends on $q$ but not on $t$.
    Now, fix $0\le t_1 \le \dots \le t_m \le q$ and a real vector $\Lambda = \left( \lambda_\ell \right)_{1\le \ell\le m}$.
    We can write
    \begin{equation}\label{eq: time-joint occinj as U-stat}
        \sum_{\ell=1}^m \lambda_\ell \occinj{\pi}{\tau_{n,t_\ell}}
        = \sum_{1\le k_1< \dots< k_d\le K_n} f_\Lambda\left( B^q_{k_1,\cdot}, \dots, B^q_{k_d,\cdot} \right)
    \end{equation}
    where
    \begin{equation*}
        f_\Lambda\left( B^q_{k_1,\cdot}, \dots, B^q_{k_d,\cdot} \right)
        = \sum_{\ell=1}^m \lambda_\ell \prod_{j=1}^d \occ{\pi_j}{B^q_{k_j,t_\ell}} .
    \end{equation*}
    One may thus see \eqref{eq: time-joint occinj as U-stat} as an asymmetric $U$-statistic with the underlying i.i.d.~random variables $B^q_{k,\cdot} = \left(B^q_{k,t}\right)_{t\in[0,q]}$ for $k\ge1$, and stopped at the random time $K_n$.
    Using \cite[Theorem~3.11]{J18} and \eqref{eq: approx occ occinj in Donsker}, we deduce that
    \begin{equation}\label{eq: cv distrib of time-combinations}
        \frac{ \sum_{\ell=1}^m \lambda_\ell \occ{\pi}{\tau_{n,t_\ell}} - \binom{n}{d} \sum_{\ell=1}^m \lambda_\ell \prod_{j=1}^d \frac{\expec{\occ{\pi_j}{B^q_{t_\ell}}}}{\expec{\abs{B^q}}} }{n^{d-1/2}}
        \cv{}
        \Normal{0}{\eta_\Lambda^2}
    \end{equation}
    in distribution, for some $\eta_\Lambda^2 \ge0$, where $\left( B^q_t \right)_{t\in[0,q]} \sim \ContinuousBlock{q}$.
    When applied with $m=1$, this yields
    \begin{equation*}
        \binom{n}{d}^{-1} \occ{\pi}{\tau_{n,t}}
        \,\cv{n\to\infty}\,
        \prod_{j=1}^d \frac{\expec{\occ{\pi_j}{B^q_{t}}}}{\expec{\abs{B^q}}}
    \end{equation*}
    in probability (and almost surely by \cite[Theorem~3.9]{J18}), for all $0\le t\le q< 1$.
    Using \Cref{th: LLN and CLT for joint patterns with fixed q} or the fact that $q\in[t,1)$ is arbitrary, we deduce that
    \begin{equation}\label{eq: double expression for e}
        \prod_{j=1}^d \frac{\expec{\occ{\pi_j}{B^q_{t}}}}{\expec{\abs{B^q}}}
        = \prod_{j=1}^d \frac{\expec{\occ{\pi_j}{B^t_{t}}}}{\expec{\abs{B^t}}}
        = e_{\pi,t}
    \end{equation}
    for all $0\le t\le q< 1$, which is a useful identity in itself.
    Thus \eqref{eq: cv distrib of time-combinations} becomes
    \begin{equation}\label{eq: cv distrib of time-combinations with e}
        \frac{ \sum_{\ell=1}^m \lambda_\ell \occ{\pi}{\tau_{n,t_\ell}} - \binom{n}{d} \sum_{\ell=1}^m \lambda_\ell e_{\pi,t_\ell} }{n^{d-1/2}}
        \cv{}
        \Normal{0}{\eta_\Lambda^2} .
    \end{equation}
    The asymptotic variance $\eta_\Lambda^2$ is expressed in terms of the (non-centered) univariate projections:
    \begin{equation*}
        f_{\Lambda,j}\left( (B^q_t)_{t\in[0,q]} \right) := \expec{ f_\Lambda\left(B^q_{1,\cdot},\dots,B^q_{d,\cdot}\right) \,\Big|\, B^q_{j,\cdot} = B^q_\cdot } 
        = \sum_{\ell=1}^m \lambda_\ell \frac{\occ{\pi_j}{B^q_{t_\ell}}}{\expec{\occ{\pi_j}{B^q_{t_\ell}}}} \prod_{k=1}^d \expec{\occ{\pi_k}{B^q_{t_\ell}}} ,
    \end{equation*}
    for $j\in[d]$, where we set $\frac{\occ{\rho}{B^q_{0}}}{\expec{\occ{\rho}{B^q_{0}}}} := 1$ for all $\rho\notin \{\id_r, r\ge1\}$ by convention.
    Note that $\expec{\occ{\pi}{B^q_t}}>0$ for all $\pi\in\S$ and $0< t\le q< 1$, since the support of $B^q_t$ contains the set of indecomposable permutations.
    Also, $\expec{\occ{\id_r}{B^q_0}}>0$ for all $r\in\N^*$ since the support of $B^q_0$ is the set of identity permutations.
    Then by \cite[Equation~(3.28)]{J18}:
    \begin{align*}
        \eta_\Lambda^2
        = &\expec{\abs{B^q}}^{1-2d}
        \sum_{1\le i,j\le d}
        \binom{i\p j\m 2}{i\m 1} \binom{2d\m i\m j}{d\m i} \frac{1}{(2d - 1)!}
        \cov{f_{\Lambda,j}\left( (B^q_t)_{t\in[0,q]} \right)}{f_{\Lambda,i}\left( (B^q_t)_{t\in[0,q]} \right)} \nonumber
        \\&- \frac{2\expec{\abs{B^q}}^{-2d}}{(d-1)!d!}
        \left( \sum_{\ell=1}^m \lambda_\ell \prod_{k=1}^d \expec{\occ{\pi_k}{B^q_{t_\ell}}}\right)
        \sum_{i=1}^d \cov{\abs{B^q}}{f_{\Lambda,i}\left( (B^q_t)_{t\in[0,q]} \right)} \nonumber
        \\&+ \frac{\expec{\abs{B^q}}^{-2d-1}}{(d-1)!^2}
        \left( \sum_{\ell=1}^m \lambda_\ell \prod_{k=1}^d \expec{\occ{\pi_k}{B^q_{t_\ell}}} \right)^2
        \var{\abs{B^q}} ,
    \end{align*}
    where $B_\cdot^q = \left( B^q_t \right)_{t\in[0,q]} \sim \ContinuousBlock{q}$.
    As in the proof of \Cref{th: LLN and CLT for joint patterns with fixed q}, we can write $\eta_\Lambda^2 = \Lambda^{\rm T} H_{\pi} \Lambda$ where for all $0\le s,t\le q$:
    \begin{align}\label{eq: formula covariance function Donsker}
        H_\pi(s,t)
        = \expec{\abs{B^q}} e_{\pi,s} e_{\pi,t} \Bigg(
        &\frac{1}{(2d\m1)!}
        \sum_{1\le i,j\le d}
        \binom{i\p j\m 2}{i\m 1} \binom{2d\m i\m j}{d\m i} \cov{ \frac{\occ{\pi_j}{B^q_t}}{\expec{\occ{\pi_j}{B^q_t}}} }{ \frac{\occ{\pi_i}{B^q_s}}{\expec{\occ{\pi_i}{B^q_s}}} } \nonumber
        \\&- \frac{1}{(d\m1)!d!}
        \sum_{i=1}^d \cov{ \frac{\abs{B_\cdot^q}}{\expec{\abs{B_\cdot^q}}} }{ \frac{\occ{\pi_i}{B^q_s}}{\expec{\occ{\pi_i}{B^q_s}}} + \frac{\occ{\pi_i}{B^q_t}}{\expec{\occ{\pi_i}{B^q_t}}} } \nonumber
        \\&+ \frac{1}{(d\m1)!^2}
        \var{ \frac{\abs{B^q}}{\expec{\abs{B^q}}} } 
        \Bigg) ,
    \end{align}
    where $\frac{\occ{\rho}{B^q_{0}}}{\expec{\occ{\rho}{B^q_{0}}}} := 1$ for all $\rho\notin \{\id_r, r\ge1\}$ by convention.
    The identity $\Lambda^{\rm T} H_{\pi} \Lambda = \eta_\Lambda^2 \ge0$ for all vectors $\Lambda$ shows that $H_\pi$ is positive semidefinite, thus there exists a centered Gaussian process $\left( X_{\pi,t} \right)_{t\in[0,q]}$ with covariance function $H_\pi$.
    Since \eqref{eq: cv distrib of time-combinations} holds for all $\Lambda = (\lambda_\ell)_{1\le \ell\le m} \in\R^m$, the Cramér--Wold theorem implies that the random vector
    \begin{equation*}
        \left(
        \frac{\occ{\pi}{\tau_{n,t_\ell}} - \binom{n}{d} e_{\pi,t_\ell} }{n^{d-1/2}}
        \right)_{1\le \ell\le m}
    \end{equation*}
    converges in distribution, as $n\to\infty$, to a Gaussian vector with covariance matrix $\left( H_\pi(t_\ell, t_{\ell'}) \right)_{1\le \ell,\ell'\le m}$.
    Since this holds for any $0\le t_1\le \dots\le t_m\le q< 1$, this proves the finite-dimensional convergence in distribution of $\frac{\occ{\pi}{\tau_{n,\cdot}} - \binom{n}{d} e_{\pi,\cdot} }{n^{d-1/2}}$ to $X_\pi$, and also that $H_\pi$ does not depend on $q$ (hence $X_{\pi,\cdot}$ can be extended to $[0,1)$).
    Convergence of joint moments holds as usual, by \cite[Theorem~3.17]{J18}, \eqref{eq: approx occ occinj in Donsker} and Hölder's inequality.
\end{proof}

\subsection{Continuity of the limit Gaussian process}\label{sec: proof continuity}

The goal of this section is to prove \Cref{prop: continuity}, using in particular the Kolmogorov continuity theorem.
This requires moment estimates on $\occ{\pi}{B_t^q} - \occ{\pi}{B_s^q}$ for all $\pi\in\S$ and $0< s\le t\le q< 1$, which we obtain by comparison with the number of jumps in a continuous-time infinite Mallows process.

In what follows, we write $\mu^{*k}$ for the convolution of a measure $\mu$ with itself $k$ times. If $\mu$ is a probability distribution, then $\mu^{*k}$ is the distribution of the sum of $k$ i.i.d.~random variables with law $\mu$.

\begin{lemma}\label{lem: increment law}
    Let $s\in[0,1)$ and $k\in\N^*$.
    Let $\left( R_t \right)_{t\in[s,1)}$ be a birth process started at $R_s=k$, with birth rates $p(j,t)=\frac{j}{1-t}$ for $j\in\N^*$, $t\in[s,1)$.
    Its marginal distributions are given by $R_t \sim \Geom{\frac{1-t}{1-s}}^{*k}$ for all $t\in[s,1)$.
\end{lemma}

\begin{proof}
    For all $t\in[s,1)$, write $\mu_{k,s,t}$ for the law of $R_t$, that is $\mu_{k,s,t}(j) := \probcond{R_t = j}{R_s=k}$ for all $j\in\N$.
    Clearly, $\mu_{k,s,t}(j) = 0$ whenever $j<k$, and $\mu_{k,s,s} = \delta_k$.
    Moreover, the Chapman--Kolmogorov forward equations (see e.g.~the proof of \cite[Proposition~3.5]{AK24}) yield:
    \[
    \partial_t \mu_{k,s,t}(j) = \mu_{k,s,t}(j-1) \frac{j-1}{1-t} - \mu_{k,s,t}(j) \frac{j}{1-t}
    \]
    for every $j\in\N^*$ and $t\in[s,1)$.
    One can check that the (only) solution of this system is given by
    \[
    \mu_{k,s,t}(j) = \binom{j-1}{k-1} \left(\frac{t-s}{1-s}\right)^{j-k} \left(\frac{1-t}{1-s}\right)^k ,
    \]
    for all $j\ge k$, which is indeed the $\Geom{\frac{1-t}{1-s}}^{*k}$ distribution.
\end{proof}

\begin{remark}
    The expression $\mu^{*k}$ of the distribution of $R_t$ in \Cref{lem: increment law} can be related to a \enquote{branching property}; it is a simple consequence of the fact that the birth rates are of the form $p(j,t) = j \cdot \tilde p(t)$.
\end{remark}

\begin{corollary}\label{cor: increment moments}
    Let $\left( R_t \right)_{t\in[s,1)}$ be a birth process started at $R_0=1$, with birth rates $p(j,t)=\frac{j}{1-t}$ for $j\in\N^*$, $t\in[0,1)$.
    For all $0\le s\le t< 1$, we have:
    \[
    \expec{ R_t - R_s } = \frac{t-s}{(1-s)(1-t)}
    \quad;\quad
    \var{ R_t - R_s } = \frac{(t-s)( 1 + st - 2s )}{(1-s)^2(1-t)^2} .
    \]
\end{corollary}

\begin{proof}
    Recall that $R_t \sim \Geom{1-t}$ for all $t\in[0,1)$ by \Cref{lem: geometric birth process}, therefore:
    \[
    \expec{R_t-R_s} = \frac{1}{1-t} - \frac{1}{1-s} = \frac{t-s}{(1-s)(1-t)} .
    \]
    Furthermore, by \Cref{lem: increment law}, the conditional law of $R_t$ given $R_s$ is $\Geom{\frac{1-t}{1-s}}^{*R_s}$.
    Hence, by the law of total variance:
    \begin{align*}
        \var{ R_t - R_s }
        &= \var{ \expec{ R_t - R_s \,\big|\, R_s} } 
        + \expec{ \var{ R_t - R_s \,\big|\, R_s} }
        \\&= \var{ R_s \frac{1-s}{1-t} - R_s }
        + \expec{ R_s \frac{(t-s)(1-s)}{(1-t)^2} }
        \\&= \frac{s}{(1-s)^2} \left(\frac{t-s}{1-t}\right)^2
        + \frac{1}{1-s} \frac{(t-s)(1-s)}{(1-t)^2}
        \\&= \frac{(t-s)( 1 + st - 2s )}{(1-s)^2(1-t)^2} ,
    \end{align*}
    as announced.
\end{proof}

\begin{lemma}\label{lem: Lipschitz jumps in q-block}
    Let $q\in(0,1)$.
    There exists a constant $c = c_q > 0$ such that for all $0\le s\le t\le q$:
    \[
    \expec{ \inv{B_t^q} - \inv{B_s^q} } \le c(t-s)
    \quad;\quad
    \expec{ \left(\big. \inv{B_t^q} - \inv{B_s^q} \right)^2 } \le c(t-s) \,,
    \]
    where $B_\cdot^q \sim \ContinuousBlock{q}$.
\end{lemma}

\begin{proof}
    Let $\Sigma_\cdot \sim \ContinuousMallows{\N^*}$, and let $\left( R_{i,\cdot} \right)_{i\ge1}$ be its sequence of right-inversion counts.
    For $n\in\N^*$, define: 
    \[
    J_{n,s,t} := \sum_{i=1}^n \left( R_{i,t} - R_{i,s} \right)
    = \inv{\big.\pat{[n]}{\Sigma_t}} - \inv{\big.\pat{[n]}{\Sigma_s}} .
    \]
    On the one hand, since this is a sum of i.i.d.~random variables, we can use \Cref{cor: increment moments} to compute its moments:
    \begin{equation}\label{eq: asymptotic moments J}
        \expec{J_{n,s,t}} = \frac{n(t-s)}{(1-s)(1-t)} \le \frac{n(t-s)}{(1-q)^2}
        \quad;\quad
        \var{J_{n,s,t}} = \frac{n(t-s)(1+st-2s)}{(1-s)^2(1-t)^2} \le \frac{2n(t-s)}{(1-q)^4} .
    \end{equation}
    On the other hand, recall that $\Sigma_\cdot = B_{1,\cdot}^q \oplus \dots \oplus B_{k,\cdot}^q \oplus \dots$ by \Cref{th: CTI Mallows process as sum of blocks}, where $\left( B_{k,\cdot}^q \right)_{k\ge1}$ is a sequence of i.i.d.~$q$-block processes.
    Let $K_n := \max\left\{ k\ge0 \,:\, \abs{B_{1,\cdot}^q}+ \dots+ \abs{B_{k,\cdot}^q} \le n \right\}$ and 
    \[
    J_{n,s,t}' := \sum_{k=1}^{K_n} \left( \inv{B_{k,t}^q} - \inv{B_{k,s}^q} \right) .
    \]
    Then $J_{n,s,t} = J_{n,s,t}' + \sum_{i = \abs{B_{1,\cdot}^q}+ \dots+ \abs{B_{K_n,\cdot}^q} +1}^n \left( R_{i,t} - R_{i,s} \right)$, and thus
    \begin{equation}\label{eq: approx J J'}
        \abs{ J_{n,s,t} - J_{n,s,t}' }
        \le \inv{B_{K_n+1,t}^q} - \inv{B_{K_n+1,s}^q}
        \le \abs{B_{K_n+1,\cdot}^q}^2 ,
    \end{equation}
    which is bounded over $n$ in $L^p$ for all $p>0$, as in the proof of \Cref{th: LLN and CLT for joint patterns with fixed q}.
    Since $J_{n,s,t}'$ is a $U$-statistic stopped at a renewal time, we can apply \cite[Theorem~3.17]{J18}:
    \begin{equation}\label{eq: asymptotic expec J'}
        \expec{ J_{n,s,t}' } \sim n \frac{ \expec{\inv{B_t^q} - \inv{B_s^q}} }{ \expec{\abs{B_\cdot^q}} }
    \end{equation}
    as $n\to\infty$, where $B_\cdot^q \sim \ContinuousBlock{q}$, and
    \begin{multline}\label{eq: asymptotic var J'}
        \var{ J_{n,s,t}' } \sim n \Bigg(
        \frac{ \var{\inv{B_t^q} - \inv{B_s^q}} }{ \expec{\abs{B_\cdot^q}} }
        - 2 \frac{ \expec{\inv{B_t^q} - \inv{B_s^q}} \cov{\abs{B_\cdot^q}}{\inv{B_t^q} - \inv{B_s^q}} }{ \expec{\abs{B_\cdot^q}}^2 }
        \\+ \frac{ \expec{\inv{B_t^q} - \inv{B_s^q}}^2 \var{\abs{B_\cdot^q}} }{ \expec{\abs{B_\cdot^q}}^3 }
        \Bigg) .
    \end{multline}
    From \Cref{eq: asymptotic moments J,eq: approx J J',eq: asymptotic expec J'}, we deduce that:
    \[
    \expec{\inv{B_t^q} - \inv{B_s^q}}
    = \expec{\abs{B_\cdot^q}} \frac{t-s}{(1-s)(1-t)}
    \le c_q (t-s)
    \]
    for some constant $c_q > 0$.
    Subsequently, the last term in \eqref{eq: asymptotic var J'} is bounded by $\O(1) (t-s)^2 \le \O(1) (t-s)$.
    Then, using $\inv{B_t^q} - \inv{B_s^q} \le \abs{B_\cdot^q}^2$, the second term in \eqref{eq: asymptotic var J'} is bounded by $\O(1) \expec{\inv{B_t^q} - \inv{B_s^q}} \le \O(1) (t-s)$.
    Therefore, by identifying \eqref{eq: asymptotic var J'} with \eqref{eq: asymptotic moments J} modulo \eqref{eq: approx J J'}, we obtain:
    \[
    \var{\inv{B_t^q} - \inv{B_s^q}} \le c_q' (t-s)
    \]
    for some constant $c_q' > 0$.
    Finally, we can write
    \[
    \expec{\left(\big. \inv{B_t^q} - \inv{B_s^q} \right)^2} \le c_q'(t-s) + c_q^2(t-s)^2 \le \left( c_q' + c_q^2 \right) (t-s)
    \]
    to conclude the proof.
\end{proof}

\begin{lemma}\label{lem: occ control from inv control}
    Let $\sigma, \sigma' \in\S_n$ such that for all $i\in[n]$, $r_i(\sigma) \le r_i(\sigma')$.
    Then, for all $\rho\in\S_r$:
    \[
    \abs{ \occ{\rho}{\sigma} - \occ{\rho}{\sigma'} }
    \le 2 \left( \inv{\sigma'} - \inv{\sigma} \right) n^{r-1} .
    \]
\end{lemma}

\begin{proof}
    By induction, it suffices to prove the inequality when $\inv{\sigma'} = \inv{\sigma} + 1$.
    In that case, there exists $i_0$ such that $r_{i_0}(\sigma') = r_{i_0}(\sigma)+1$ and $r_{i}(\sigma') = r_{i}(\sigma)$ for all $i\ne i_0$.
    It is easily checked that $\sigma'$ is then obtained from $\sigma$ by swapping $\sigma(i_0)$ with $\sigma(j_0)$, where $\sigma(j_0) = \min\left\{ \sigma(j) : j>i_0 , \sigma(j) > \sigma(i_0) \right\}$.
    Consequently:
    \[
    \abs{ \occ{\rho}{\sigma} - \occ{\rho}{\sigma'} }
    \le \card{ I \in \binom{[n]}{r} \,:\, I \cap \{i_0,j_0\} \ne \emptyset }
    \le 2 n^{r-1}
    \]
    as announced.
\end{proof}

\begin{corollary}\label{cor: Holder e}
    For all $\pi\in\S$, $q\in(0,1)$, $\gamma\in(0,1)$, there exists $c>0$ such that:
    \[
    \text{for all }0\le s\le t\le q,\quad
    \abs{e_{\pi,t} - e_{\pi,s}} \le c(t-s)^\gamma .
    \]
\end{corollary}

\begin{proof}
    We can assume w.l.o.g.~that $\gamma\ge1/2$.
    Recall that
    $e_{\pi,t} =\prod_{j=1}^d \frac{\expec{\occ{\pi_j}{B^q_{t}}}}{\expec{\abs{B^q}}}$
    for all $t\in[0,q]$ by \Cref{eq: double expression for e}.
    Therefore it suffices to prove that for all $\rho\in\S$, the function $t\in[0,q] \mapsto \expec{\occ{\rho}{B^q_{t}}}$ is $\gamma$-H\"older-continuous.
    Using \Cref{lem: occ control from inv control} and H\"older's inequality, we can write for all $0\le s\le t\le q$:
    \begin{multline*}
        \abs{\big. \expec{\occ{\rho}{B^q_{t}}} - \expec{\occ{\rho}{B^q_{s}}} }
        \le \expec{ 2 \left(\big. \inv{B^q_{t}} - \inv{B^q_{s}} \right) \abs{B_\cdot^q}^{\abs{\rho}-1} }
        \\\le 2 \expec{ \left(\big. \inv{B^q_{t}} - \inv{B^q_{s}} \right)^{1/\gamma} }^\gamma \expec{ \abs{B^q}^{(\abs{\rho}-1)/(1-\gamma)} }^{1-\gamma} .
    \end{multline*}
    Then, using the inequality $\left(\big. \inv{B^q_{t}} - \inv{B^q_{s}} \right)^{1/\gamma} \le \left(\big. \inv{B^q_{t}} - \inv{B^q_{s}} \right)^2$, \Cref{lem: Lipschitz jumps in q-block} and the fact that $\abs{B^q}$ has finite moments of all orders, we deduce that:
    \begin{equation*}
        \abs{\big. \expec{\occ{\rho}{B^q_{t}}} - \expec{\occ{\rho}{B^q_{s}}} }
        \le c (t-s)^\gamma
    \end{equation*}
    for some $c>0$.
    This concludes the proof.
\end{proof}

\begin{corollary}\label{cor: Holder H}
    For all $\pi\in\S$, $0< \epsilon< q< 1$, and $\gamma\in(0,1)$, there exists $c>0$ such that:
    \[
    \text{for all }\epsilon\le s, t, u\le q,\quad
    \abs{H_{\pi}(t,u) - H_{\pi}(s,u)} \le c\abs{t-s}^\gamma ,
    \]
    where $H_\pi$ is defined by \eqref{eq: formula covariance function Donsker}.
\end{corollary}

\begin{proof}
    We can assume w.l.o.g.~that $\gamma\ge1/2$.
    Let $B_\cdot^q \sim \ContinuousBlock{q}$.
    Write $\pi= \pi_1\oplus \dots\oplus \pi_d$ as a sum of blocks, and set $\pi_0 := 1$.
    Then for each $0\le i,j\le d$, define 
    \[
    \text{for all }0\le s,t< 1,\quad
    f_{i,j}(s,t) := \cov{ \frac{\occ{\pi_i}{B_s^q}}{\expec{\occ{\pi_i}{B_s^q}}} }{ \frac{\occ{\pi_j}{B_t^q}}{\expec{\occ{\pi_j}{B_t^q}}} } 
    \]
    where $\frac{\occ{\rho}{B_0^q}}{\expec{\occ{\rho}{B_0^q}}} := 1$ for all $\rho\notin\{\id_r, r\ge1\}$.
    Since $\occ{\pi_0}{B} = \abs{B}$, \Cref{eq: formula covariance function Donsker} rewrites as:
    \begin{equation}\label{eq: simplified expression H}
        \text{for all }0\le s,t< 1,\quad
        H_\pi(s,t) = \sum_{0\le i,j\le d} c_{i,j,\pi,q} \, e_{\pi,s} \, e_{\pi,t} \, f_{i,j}(s,t)
    \end{equation}
    for some constants $c_{i,j,\pi,q}>0$.
    Therefore it suffices to prove that for each fixed $0\le i,j\le d$, there exists $c>0$ such that:
    \[
    \text{for all }\epsilon\le s, t, u\le q,\quad
    \abs{\big. e_{\pi,u} e_{\pi,t} f_{i,j}(t,u) - e_{\pi,u} e_{\pi,s} f_{i,j}(s,u) } \le c\abs{t-s}^\gamma .
    \]
    Since $\abs{B^q}$ has finite moments of all orders, we can bound:
    \begin{multline}\label{eq: bound increment H}
        \abs{\big. e_{\pi,u} e_{\pi,t} f_{i,j}(t,u) - e_{\pi,u} e_{\pi,s} f_{i,j}(s,u) }
        \le \abs{\big. e_{\pi,u} \left(\big. e_{\pi,t} - e_{\pi,s} \right) f_{i,j}(t,u) }
        + \abs{\big. e_{\pi,u} e_{\pi,s} \left(\big. f_{i,j}(t,u) - f_{i,j}(s,u) \right)}
        \\\le \O(1) \abs{\big. e_{\pi,t} - e_{\pi,s} }
        + \O(1) \abs{\big. e_{\pi_j,u} e_{\pi_i,s} \left(\big. f_{i,j}(t,u) - f_{i,j}(s,u) \right)}
    \end{multline}
    where the constants inside the $\O(1)$'s only depend on $\pi$ and $q$.
    The first term of \eqref{eq: bound increment H} is bounded by $\O(1) \abs{t-s}^\gamma$ by \Cref{cor: Holder e}.
    Then, let $\alpha,\bar\alpha\in(1,\infty)$ be arbirtrary H\"older conjugates.
    Recall that $\pi_i$ and $\pi_j$ are indecomposable, and use \eqref{eq: double expression for e} to write:
    \begin{align}\label{eq: bound increment cov}
        &\abs{\big. e_{\pi_j,u} e_{\pi_i,s} \left(\big. f_{i,j}(t,u) - f_{i,j}(s,u) \right)} \nonumber  
        \\&= e_{\pi_j,u} e_{\pi_i,s} \abs{ \cov{ \frac{\occ{\pi_j}{B_u^q}}{\expec{\occ{\pi_j}{B_u^q}}} }{ \frac{\occ{\pi_i}{B_t^q}}{\expec{\occ{\pi_i}{B_t^q}}} - \frac{\occ{\pi_i}{B_s^q}}{\expec{\occ{\pi_i}{B_s^q}}} } } \nonumber
        \\&\le \expec{ e_{\pi_j,u} \abs{ \frac{\occ{\pi_j}{B_u^q}}{\expec{\occ{\pi_j}{B_u^q}}} - 1 } \cdot e_{\pi_i,s} \abs{ \frac{\occ{\pi_i}{B_t^q}}{\expec{\occ{\pi_i}{B_t^q}}} - \frac{\occ{\pi_i}{B_s^q}}{\expec{\occ{\pi_i}{B_s^q}}} } } \nonumber
        \\&\le \expec{ \abs{ \frac{\occ{\pi_j}{B_u^q}}{\expec{\abs{B^q}}} - e_{\pi_j,u} }^{\bar\alpha} }^{1/\bar\alpha} \cdot e_{\pi_i,s} \, \expec{ \abs{ \frac{\occ{\pi_i}{B_t^q}}{\expec{\occ{\pi_i}{B_t^q}}} - \frac{\occ{\pi_i}{B_s^q}}{\expec{\occ{\pi_i}{B_s^q}}} }^\alpha }^{1/\alpha} .
    \end{align}
    The first term of \eqref{eq: bound increment cov} is a $\O(1)$.
    For the second term, by Minkowski's inequality we have:
    \begin{multline*}
        \expec{ \abs{ \frac{\occ{\pi_i}{B_t^q}}{\expec{\occ{\pi_i}{B_t^q}}} - \frac{\occ{\pi_i}{B_s^q}}{\expec{\occ{\pi_i}{B_s^q}}} }^\alpha }^{1/\alpha}
        \\\le \expec{ \abs{ \frac{\occ{\pi_i}{B_t^q}}{\expec{\occ{\pi_i}{B_t^q}}} - \frac{\occ{\pi_i}{B_t^q}}{\expec{\occ{\pi_i}{B_s^q}}} }^\alpha }^{1/\alpha}
        + \expec{ \abs{ \frac{\occ{\pi_i}{B_t^q}}{\expec{\occ{\pi_i}{B_s^q}}} - \frac{\occ{\pi_i}{B_s^q}}{\expec{\occ{\pi_i}{B_s^q}}} }^\alpha }^{1/\alpha} .
    \end{multline*}
    First use \eqref{eq: double expression for e} to write:
    \begin{multline*}
        e_{\pi_i,s} \expec{ \abs{ \frac{\occ{\pi_i}{B_t^q}}{\expec{\occ{\pi_i}{B_t^q}}} - \frac{\occ{\pi_i}{B_t^q}}{\expec{\occ{\pi_i}{B_s^q}}} }^\alpha }^{1/\alpha}
        \le \O(1)
        e_{\pi_i,s} \abs{\big. \frac{1}{\expec{\occ{\pi_i}{B_t^q}}} - \frac{1}{\expec{\occ{\pi_i}{B_s^q}}} } 
        \\\le \frac{\O(1)}{\expec{\occ{\pi_i}{B_t^q}}} \abs{\big. \expec{\occ{\pi_i}{B_s^q}} - \expec{\occ{\pi_i}{B_t^q}} } 
    \end{multline*}
    which is bounded by $\O(1) \abs{t-s}^\gamma$ for $\epsilon\le s\le t\le q$ by \Cref{cor: Holder e}, where the constant inside $\O(1)$ only depends on $\pi$, $\gamma$, $\epsilon$ and $q$ (this is the only step at which $\epsilon$ is relevant).
    Then let $\beta,\bar\beta\in(1,\infty)$ be arbitrary H\"older conjugates, and use \Cref{lem: occ control from inv control} to write:
    \begin{align*}
        &e_{\pi_i,s} \expec{ \abs{ \frac{\occ{\pi_i}{B_t^q}}{\expec{\occ{\pi_i}{B_s^q}}} - \frac{\occ{\pi_i}{B_s^q}}{\expec{\occ{\pi_i}{B_s^q}}} }^\alpha }^{1/\alpha} 
        \\&= \frac{1}{\expec{\abs{B^q}}} \expec{ \abs{\big. \occ{\pi_i}{B_t^q} - \occ{\pi_i}{B_s^q} }^\alpha }^{1/\alpha}
        \\&\le \O(1)\, \expec{ \left(\big. \inv{B^q_{t}} - \inv{B^q_{s}} \right)^\alpha \abs{B_\cdot^q}^{\alpha(\abs{\pi_i}-1)} }^{1/\alpha}
        \\&\le \O(1) \expec{ \left(\big. \inv{B^q_{t}} - \inv{B^q_{s}} \right)^{\alpha\beta} }^{1/(\alpha\beta)}
        \expec{ \abs{B_\cdot^q}^{\alpha\bar\beta(\abs{\pi_i}-1)} }^{1/(\alpha\bar\beta)} .
    \end{align*}
    If we choose $\alpha, \beta$ to satisfy $\alpha\beta = 1/\gamma$, then this is bounded by $\O(1) \abs{t-s}^\gamma$ by \Cref{lem: Lipschitz jumps in q-block}.
    We deduce that \eqref{eq: bound increment cov} is bounded by $\O(1) \abs{t-s}^\gamma$, where the constant inside $\O(1)$ only depends on $\pi$, $\gamma$, $\epsilon$ and $q$, therefore such a bound holds for \eqref{eq: bound increment H}.
    By \eqref{eq: simplified expression H}, this concludes the proof.
\end{proof}

The proof of \Cref{prop: continuity} now follows readily.

\begin{proof}[Proof of \Cref{prop: continuity}]
    The first item is just \Cref{cor: Holder e}.
    
    For the second item, we study the moments of $X_{\pi,t} - X_{\pi,s}$ for fixed $0< \epsilon< q< 1$ and all $\epsilon\le s\le t\le q$.
    The random variable $X_{\pi,t} - X_{\pi,s}$ follows the $\Normal{0}{\big. H_{\pi}(t,t) + H_{\pi}(s,s) - 2H_{\pi}(s,t)}$ distribution, thus for all $p>0$:
    \begin{multline*}
        \expec{ \abs{X_{\pi,t} - X_{\pi,s}}^p }
        = c_p \left(\big. H_{\pi}(t,t) + H_{\pi}(s,s)-2H_{\pi}(s,t) \right)^{p/2} 
        \\\le 2^p c_p \left( \left|\big. H_{\pi}(t,t) - H_{\pi}(s,t) \right|^{p/2} + \left|\big. H_{\pi}(s,t) - H_{\pi}(s,s) \right|^{p/2} \right) ,
    \end{multline*}
    where $c_p = \expec{ \abs{Y}^p }$ for $Y\sim\Normal{0}{1}$.
    Let $\gamma\in(0,1)$.
    Using \Cref{cor: Holder H}, we can write
    \begin{equation*}
        \expec{ \abs{ X_{\pi,t} - X_{\pi,s} }^p } \le c' (t-s)^{\gamma p/2} 
    \end{equation*}
    for some $c'>0$ and all $\epsilon\le s\le t\le q$.
    Choosing $p$ such that $\gamma p / 2 > 1$ allows us to apply the Kolmogorov continuity theorem (see e.g.~\cite[Theorem~4.23]{Kallenberg_2006}):
    there exists a version of $\left(X_{\pi,t}\right)_{t\in[\epsilon,q]}$ that is locally $\tilde \gamma$-H\"older-continuous, for every $\tilde \gamma < \frac{\tfrac12 \gamma p -1}{p}$.
    Since this holds for any $0<\epsilon<q<1$, $\gamma\in(0,1)$ and large enough $p$, this concludes the proof.
\end{proof}

\section*{Appendix: the number of inversions}\label{sec: appendix 1}

In this appendix, we provide a proof of \Cref{th: inv CLT if above thermo}.
Note that the same techniques can be used to study the asymptotics of $\inv{\tau_n}$ for different regimes of $q=q_n$.

\begin{proof}[Proof of \Cref{th: inv CLT if above thermo}]
    Using \Cref{lem: left- and right-inversion numbers of finite Mallows}, we can write
    \begin{equation}\label{eq: inversions as sum of independent}
        \inv{\tau_n} = \sum_{k=1}^n L_k - n 
    \end{equation}
    where $L_k \sim \TGeom{k}{q}$ are independent random variables.
    According to \cite[Property~4]{DR00}, we have:
    \begin{equation*}
        \expec{\inv{\tau_{n+1}}} = \frac{nq}{1\m q} - \sum_{j=2}^{n+1} \frac{j q^j}{1\m q^j}
        \quad\text{and}\quad
        \var{\inv{\tau_{n+1}}} = \frac{nq}{(1\m q)^2} - \sum_{j=2}^{n+1} \frac{j^2 q^j}{(1\m q^j)^2} .
    \end{equation*}
    The rest of the proof is standard computation.
    Since $n(1\m q)\to0$, we can write uniformly in $j\in[n+1]$:
    \begin{align*}
        q^j &= \exp\left[ j \log\left( 1-(1\m q) \right) \right]\\
        &= \exp\left[ j \left( -(1\m q) - \tfrac12(1\m q)^2 - \tfrac13(1\m q)^3 \cdot (1\p o(1)) \right) \right]\\
        &= \exp\left[ -j(1\m q) - \tfrac12j(1\m q)^2 - \tfrac13j(1\m q)^3 \cdot (1\p o(1)) \right]\\
        &= 1 -j(1\m q) - \tfrac12j(1\m q)^2 - \tfrac13j(1\m q)^3 \cdot (1\p o(1)) + \tfrac12\left( j(1\m q) + \tfrac12j(1\m q)^2 + \tfrac13j(1\m q)^3 \cdot (1\p o(1)) \right)^2\\
        &\quad - \tfrac16\left( j(1\m q) + \tfrac12j(1\m q)^2 + \tfrac13j(1\m q)^3 \cdot (1\p o(1)) \right)^3 \cdot (1\p o(1))\\
        &= 1 - j(1\m q) + (1\m q)^2\left(\tfrac12j^2 - \tfrac12j\right) + (1\m q)^3\left(-\tfrac16j^3 + \tfrac12j^2 - \tfrac13j\right) (1\p o(1)) .
    \end{align*}
    Then:
    \begin{align*}
        \frac{j q^j}{1\m q^j} &= j\left( 1 - j(1\m q) + (1\m q)^2\left(\tfrac12j^2 - \tfrac12j\right) (1\p o(1)) \right)\\
        &\quad \cdot\left( j(1\m q) - (1\m q)^2\left(\tfrac12j^2 - \tfrac12j\right) - (1\m q)^3\left(-\tfrac16j^3 + \tfrac12j^2 - \tfrac13j\right) (1\p o(1)) \right)^{-1}\\
        &= j\left( 1 - j(1\m q) + (1\m q)^2\left(\tfrac12j^2 - \tfrac12j\right) (1\p o(1)) \right)\\
        &\quad \cdot j^{-1}(1\m q)^{-1} \cdot \left( 1 - (1\m q)\left(\tfrac12j - \tfrac12\right) - (1\m q)^2\left(-\tfrac16j^2 + \tfrac12j - \tfrac13\right) (1\p o(1)) \right)^{-1}\\
        &= \left( (1\m q)^{-1} - j + (1\m q)\left(\tfrac12j^2 - \tfrac12j\right) (1\p o(1)) \right)\\
        &\quad \cdot \left( 1 + (1\m q)\left(\tfrac12j - \tfrac12\right) + (1\m q)^2\left(-\tfrac16j^2 + \tfrac12j - \tfrac13\right) (1\p o(1)) + (1\m q)^2\left(\tfrac12j - \tfrac12\right)^2 \cdot (1\p o(1)) \right)\\
        &= \left( (1\m q)^{-1} - j + (1\m q)\left(\tfrac12j^2 - \tfrac12j\right) (1\p o(1)) \right)\\
        &\quad \cdot \left( 1 + (1\m q)\left(\tfrac12j - \tfrac12\right) + (1\m q)^2\left(\tfrac{1}{12}j^2 - \tfrac{1}{12}\right) (1\p o(1)) \right)\\
        &= (1\m q)^{-1} + \left(-\tfrac12j-\tfrac12\right) + (1\m q)\left(\tfrac{1}{12}j^2 - \tfrac{1}{12}\right)(1\p o(1)) .
    \end{align*}
    We can deduce the desired estimate for the mean:
    \begin{align*}
        \expec{\inv{\tau_{n+1}}} &= \frac{n}{1\m q} - n - \sum_{j=2}^{n+1} (1\m q)^{-1} + \left(-\tfrac12j-\tfrac12\right) + (1\m q)\left(\tfrac{1}{12}j^2 - \tfrac{1}{12}\right)(1\p o(1))\\
        &= \tfrac14n^2 - \tfrac{1}{36}(1\m q)n^3 + \Landau{o}{}{n^{3/2} \vee (1\m q)n^3} .
    \end{align*}
    Now, let us turn to the variance.
    We have, uniformly in $j\in[n+1]$:
    \begin{align*}
        \frac{j^2 q^j}{(1\m q^j)^2} 
        &= j^2\left( 1 - j(1\m q) + (1\m q)^2\left(\tfrac12j^2 - \tfrac12j\right) (1\p o(1)) \right)\\
        &\quad \cdot\left( j(1\m q) - (1\m q)^2\left(\tfrac12j^2 - \tfrac12j\right) - (1\m q)^3\left(-\tfrac16j^3 + \tfrac12j^2 - \tfrac13j\right) (1\p o(1)) \right)^{-2}\\
        &= j^2\left( 1 - j(1\m q) + (1\m q)^2\left(\tfrac12j^2 - \tfrac12j\right) (1\p o(1)) \right)\\
        &\quad \cdot j^{-2}(1\m q)^{-2} \cdot \left( 1 - (1\m q)\left(\tfrac12j - \tfrac12\right) - (1\m q)^2\left(-\tfrac16j^2 + \tfrac12j - \tfrac13\right) (1\p o(1)) \right)^{-2}\\
        &= \left( (1\m q)^{-2} - j(1\m q)^{-1} + \left(\tfrac12j^2 - \tfrac12j\right) (1\p o(1)) \right)\\
        &\quad \cdot \left( 1 + 2(1\m q)\left(\tfrac12j - \tfrac12\right) + 2(1\m q)^2\left(-\tfrac16j^2 + \tfrac12j - \tfrac13\right) (1\p o(1)) + 3(1\m q)^2\left(\tfrac12j - \tfrac12\right)^2 \cdot (1\p o(1)) \right)\\ 
        &= \left( (1\m q)^{-2} - j(1\m q)^{-1} + \left(\tfrac12j^2 - \tfrac12j\right) (1\p o(1)) \right)\\
        &\quad \cdot \left( 1 + (1\m q)\left(j - 1\right) + (1\m q)^2\left(\tfrac{5}{12}j^2 - \tfrac12j + \tfrac{1}{12}\right) (1\p o(1)) \right)\\
        &= (1\m q)^{-2} - (1\m q)^{-1} + \left(-\tfrac{1}{12}j^2 + \tfrac{1}{12}\right)(1\p o(1)) + \Landau{\O}{}{(1\m q)j^3} ,
    \end{align*}
    therefore:
    \begin{align*}
        \var{\inv{\tau_{n+1}}} 
        &= \frac{n}{(1\m q)^2} - \frac{n}{(1\m q)} - \sum_{j=2}^{n+1} (1\m q)^{-2} - (1\m q)^{-1} + \left(-\tfrac{1}{12}j^2 + \tfrac{1}{12}\right)(1\p o(1)) + \Landau{\O}{}{(1\m q)j^3} \\
        &= \tfrac{1}{36}n^3 + \Landau{o}{}{n^3} .
    \end{align*}
    It remains to deduce the announced asymptotic normality; 
    it suffices to prove it for $\sum_{k=1}^n L_k$.
    Since this is a sum of independent variables, it suffices to check Lyapunov's condition.
    For any $\delta>0$, we have:
    \begin{align*}
        \frac{1}{\var{\inv{\tau_n}}^{1+\delta/2}} \sum_{k=1}^n \expec{\abs{L_k - \E L_k}^{2+\delta}}
        \le \frac{\O(1)}{n^{3+3\delta/2}} \sum_{k=1}^n k^{2+\delta}
        \le \frac{\O(1)}{n^{3+3\delta/2}} \, n^{3+\delta}
        \cv{n\to\infty} 0 .
    \end{align*}
    We can thus conclude with Lyapunov's CLT.
\end{proof}

\section*{Acknowledgements}

The author is very grateful to his advisor Valentin F\'eray for many insightful discussions and valuable suggestions.

\bibliographystyle{alphaurl}
\bibliography{bibli}

\end{document}